\documentclass[10pt]{article}
\usepackage{diagbox}
\usepackage{lmodern}
\usepackage{amsmath}
\usepackage[T1]{fontenc}
\usepackage[utf8]{inputenc}
\usepackage{authblk}
\usepackage{amsfonts}
\usepackage{graphicx}
\usepackage{rotating}
\usepackage{amssymb}
\usepackage[english]{babel}
\usepackage{color}
\usepackage{amsthm}
\usepackage{graphicx}
\usepackage{tkz-euclide}
\tikzset{elegant/.style={smooth,thick,samples=50,cyan}}
\usepackage{color,tikz}
\usetikzlibrary{patterns}
\usetikzlibrary{decorations.pathreplacing,calc}
\usepackage{rotating}
\usepackage{mathrsfs}
\usepackage{makecell}
\usepackage{microtype}
\usepackage{enumerate}
\usepackage{mathscinet}
\usepackage{array}
\usepackage{multirow}
\usepackage{booktabs}
\usepackage{enumerate}
\usepackage[cal=boondoxo,bb=ams]{mathalfa}
\usepackage{hyperref}
\hypersetup{hidelinks}
\usepackage{titlesec}
\usepackage{arydshln}

\newtheorem{theorem}{Theorem}[section]

\newtheorem{prop}{Proposition}[section]

\newtheorem{remark}{Remark}[section]

\newcommand{\ml}{\mathcal}
\newcommand{\mb}{\mathbb}

\DeclareMathOperator{\intt}{int}
\DeclareMathOperator{\extt}{ext}
\DeclareMathOperator{\bdd}{bdd}

\DeclareMathOperator{\Th}{th}
\DeclareMathOperator{\diag}{diag}

\title{Asymptotic behavior of solutions for the thermoviscous acoustic systems}
\author[1]{Wenhui Chen\thanks{Wenhui Chen (wenhui.chen.math@gmail.com)}}
\affil[1]{School of Mathematics and Information Science, Guangzhou University, 510006 Guangzhou, China}
\author[2]{Yan Liu\thanks{Yan Liu (ly801221@163.com)}}
\affil[2]{Department of Applied Mathematics, Guangdong University of Finance, 510521 Guangzhou, China}

\author[3]{Mengjun Ma\thanks{Mengjun Ma (mamj7@mail2.sysu.edu.cn)}}
\author[3]{Xulong Qin\thanks{Xulong Qin (qinxul@mail.sysu.edu.cn)}}
\affil[3]{Department of Mathematics, Sun Yat-Sen University, 510275 Guangzhou, China}

\date{}
\begin{document}

\maketitle
\begin{abstract}
	\medskip
We study some asymptotic properties of solutions for the acoustic coupled systems in thermoviscous fluids which was proposed by [Karlsen-Bruus, \emph{Phys. Rev. E} (2015)]. Basing on the WKB analysis and the Fourier analysis, we derive optimal estimates and large time asymptotic profiles of the  energy term via diagonalization procedure, and of the velocity potential via reduction methodology. We found that the wave effect has a dominant influence for lower dimensions comparing with thermal-viscous effects. Moreover, by employing suitable energy methods, we rigorously demonstrate global (in time) inviscid limits as the momentum diffusion coefficient vanishes, whose limit model can be regarded as the thermoelastic acoustic systems in isotropic solids. These results explain some influence of the momentum diffusion on asymptotic behavior of solutions.
 \\
	
	
	\noindent\textbf{AMS Classification (2020)} 35G10, 35A01, 35B40, 35B20 
	
\end{abstract}
\fontsize{12}{15}
\selectfont
\section{Introduction}\label{Section-Introduction}
$\ \ \ \ $It is well-known that acoustic waves have been widely applied in medical and industrial uses of high-intensity ultra sound, for example, medical imaging and therapy, ultrasound cleaning and welding (see \cite{Abramov-1999,Dreyer-Krauss-Bauer-Ried-2000,Kaltenbacher-Landes-Hoffelner-Simkovics-2002} and references therein). To describe the propagation of sound mathematically in thermoviscous fluids, some acoustic models related to damped wave equations arise, including the Kuznetsov equation \cite{Mizohata-Ukai=1993,Kaltenbacher-Lasiecka=2012}, the Westervelt equation \cite{Kaltenbacher-Lasiecka=2009,Meyer-Wilke=2011}, the Jordan-Moore-Gibson-Thompson equation \cite{Kaltenbacher-Lasiecka-Pos-2012,Kaltenbacher-Nikolic-2019,Racke-Said-2020,B-L-2020,Said-2022,Chen-Takeda=2023}, and the Blackstock equation \cite{Brunnhuber-2015,Kaltenbacher-Thalhammer-2018,Chen-Ikehata-Palmieri=2023}. The above-mentioned mathematical models are the single scalar equations for the acoustic velocity potential, and they have been established via the Lighthill approximation of compressible Navier-Stokes equations associated with the Fourier/Cattaneo law of heat conduction under irrotational flows. Concerning their detail deductions,  the interested reader is referred to \cite{Blackstock-1963,Kuznetsov-1971,Hamilton-Blackstock-1998,Jordan-2014,Kaltenbacher-Thalhammer-2018} and references therein. Indeed, an acoustic coupled system always simulates real world phenomena precisely instead of an acoustic single equation due to the coupling effect between the acoustic velocity potential and the temperature. Nevertheless, it seems that acoustic coupled systems have not yet been studied  in depth mathematically.

Recently, the pioneering work \cite{Karlsen-Bruus=2015} applied the thermoviscous perturbation theory of acoustics in fluids to deduce first order equations for fluids, where the acoustic field was regarded as a perturbation of a constant state. Combining the  governing equations for the conservations of mass, momentum and energy, the state equation, and the first law of thermodynamics under irrotational flows, the authors of \cite{Karlsen-Bruus=2015} derived the thermoviscous acoustic coupled systems (see \cite[Equations (21a)-(21b)]{Karlsen-Bruus=2015}). The corresponding Cauchy problem of the thermoviscous acoustic systems is formulated by
\begin{align}\label{Thermo-Acoustic-System}
\begin{cases}
\phi_{tt}-c_0^2\Delta\phi-(1+\beta)\nu_0\Delta\phi_t+\alpha_pc_0^2T_t=0,&x\in\mb{R}^n,\ t>0,\\[0.26em]
\displaystyle{T_t-\gamma D_{\Th}\Delta T+\frac{\gamma-1}{\alpha_p}\Delta\phi=0},&x\in\mb{R}^n,\ t>0,\\
\phi(0,x)=\phi_0(x),\ \phi_t(0,x)=\phi_1(x),\ T(0,x)=T_0(x),&x\in\mb{R}^n,
\end{cases}
\end{align}
with the acoustic velocity potential of the longitudinal component $\phi=\phi(t,x)$ and the temperature $T=T(t,x)$, where all physical quantities are addressed in Table \ref{Table_1}. Without loss of generality, we will consider $(1+\beta)\nu_0\neq\gamma D_{\Th}$. Note that the equal case can be treated simply by following the same philosophies and procedures as those in this paper with slight modifications.
\begin{table}[http]
	\centering	
	\caption{Physical quantities in the thermoviscous acoustic systems}
	\begin{tabular}{ll}
		\toprule
		Quantity & Notation     \\
		\midrule
		Constant density of the equilibrium state & $\rho_0>0$\\
		Isothermal compressibility & $\kappa_T>0$\\
		Isothermal speed of sound  & $c_0=1/\sqrt{\rho_0\kappa_T}$ \\
		Dynamic shear/bulk viscosity & $\eta^s,\eta^b>0$\\
		Viscosity ratio  & $\beta=\eta^b/\eta^s+1/3$  \\
		Momentum diffusion constant & $\nu_0>0$ \\
		Isobaric thermal expansion coefficient & $\alpha_p>0$ \\
		Specific heat capacity at constant pressure/volume & $c_P,c_V>0$\\
		Thermodynamic identity & $\gamma=c_P/c_V>1$ \\
		Thermal diffusion constant& $D_{\Th}>0$\\
		\bottomrule
	\end{tabular}
	\label{Table_1}
\end{table}

\noindent The main purpose of this manuscript is to investigate asymptotic properties of solutions for the thermoviscous acoustic systems \eqref{Thermo-Acoustic-System}, including optimal growth/decay estimates and asymptotic profiles of solutions for large time, moreover, global (in time) inviscid limits of solutions in $L^{\infty}$ norms as the momentum diffusion coefficient vanishes. Particularly, because the momentum diffusion is basically the kinematic viscosity of the fluid causing a random and undetermined direction of molecules, we are interested in some influence of the momentum diffusion coefficient $\nu_0$ on asymptotic behavior of solutions. Simultaneously, according to the physical significance of the viscous dissipation $-(1+\beta)\nu_0\Delta\phi_t$ with $1+\beta> 4/3$ in the first equation of \eqref{Thermo-Acoustic-System}, the nomenclature ``\emph{inviscid limit}'' should be understood in the context as the vanishing momentum diffusion limit when $\nu_0\downarrow 0$.  To the best of authors' knowledge, for the sake of the coupling structure in the linear parts, the aforementioned qualitative properties for the thermoviscous acoustic systems \eqref{Thermo-Acoustic-System} are non-trivial.

Concerning large time asymptotic behavior of solutions, our contribution is twofold. For one thing, in Section \ref{Section-Decay-Energy}, by applying the refined WKB analysis and multi-step diagonalization procedure (developed in \cite{Yagdjian=1997,Reissig-Wang=2005,Yang-Wang=2006,Jachmann=2008,Reissig=2016}), we obtain optimal decay estimates for the  energy term $\Phi=\Phi(t,x)\in\mb{R}^3$ which is defined by
\begin{align}\label{Micro-energy}
	\Phi:=\big(\phi_t+i\sqrt{\gamma}c_0|D|\phi,\phi_t-i\sqrt{\gamma}c_0|D|\phi,T \big)^{\mathrm{T}}.
\end{align}
The pseudo-differential operator $|D|$ carries its symbol $|\xi|$. Furthermore, the large time profile of $\Phi$ is related to $\Psi=\Psi(t,x)$ for the first order system
\begin{align}\label{Ref-System}
\Psi_t+\diag(0,i\sqrt{\gamma}c_0,-i\sqrt{\gamma}c_0)|D|\Psi-\diag(D_{\Th},\Gamma_0,\Gamma_0)\Delta\Psi=0
\end{align}
with suitable initial data
 and $\Gamma_0:=\frac{1}{2}[(1+\beta)\nu_0+(\gamma-1)D_{\Th}]>0$. This approach did not be used in acoustic wave equations since the considered acoustic models in the past are scalar and single PDEs.  For another, by reducing the coupled systems \eqref{Thermo-Acoustic-System} to the third order (in time) evolution equation for the acoustic velocity potential $\phi$, in Section \ref{Section-Solution-Itself}, we employ the refined Fourier analysis to derive optimal growth estimates when $n=1,2$, and optimal decay estimates when $n\geqslant 3$ as follows:
\begin{align}\label{A(t)}
\|\phi(t,\cdot)\|_{L^2}\simeq\ml{A}_n(t):=\begin{cases}
	\sqrt{t}&\mbox{when}\ \ n=1,\\
	\sqrt{\ln t}&\mbox{when}\ \ n=2,\\
	t^{\frac{1}{2}-\frac{n}{4}}&\mbox{when}\ \ n\geqslant3,
\end{cases}
\end{align} 
for large time $t\gg1$. The optimality is guaranteed by the same kind of upper and lower bounds for $t\gg1$. The optimal estimates \eqref{A(t)} mean that thermal-viscous effects are not sufficient to provide decay properties for the velocity potential in the $L^2$ norm for lower dimensions. Thanks to the identical growth rates for the thermoviscous acoustic systems \eqref{Thermo-Acoustic-System} and the free wave equation \eqref{Eq-Wave} when $n=1,2$, we claim that the decisive part  in lower dimensions is the wave part, whereas thermal effect and viscous effect exert  crucial combined influence when $n\geqslant 3$ (see detail explanation in Remark \ref{Rem-Wave-structure}). It is one of the new discoveries of our manuscript. We also find an optimal leading term of the velocity potential and the optimal refined estimates
\begin{align*}
\|\phi(t,\cdot)-\varphi(t,\cdot)\|_{L^2}\simeq t^{-\frac{n}{4}}
\end{align*}
for large time $t\gg1$, where $\varphi=\varphi(t,x)$ is defined by the diffusion waves
\begin{align}\label{first-leading-term}
\varphi(t,x):=\ml{F}^{-1}_{\xi\to x}\left(\frac{\sin(\sqrt{\gamma}c_0|\xi|t)}{\sqrt{\gamma}c_0|\xi|}\mathrm{e}^{-\Gamma_0|\xi|^2t}\right)\int_{\mb{R}^n}\phi_1(x)\mathrm{d}x. 
\end{align}
As a byproduct, a second  profile for large time associated with further refined estimates   will be stated in Section \ref{Section_Main_Result}. In conclusion, large time asymptotic behavior  of solutions for the thermoviscous acoustic systems \eqref{Thermo-Acoustic-System} is determined by the diffusion waves \eqref{first-leading-term}, in which both thermal as well as momentum diffusions contribute to the diffusive part $\exp(-\Gamma_0|\xi|^2t)$ due to the definition of $\Gamma_0$.

Let us turn to asymptotic behavior of solutions with the small momentum diffusion constant $0<\nu_0\ll 1$. Taking the limit case $\nu_0=0$ formally, the thermoviscous acoustic systems \eqref{Thermo-Acoustic-System} will  turn into 
\begin{align}\label{Thermo-Acoustic-System-zero}
	\begin{cases}
		\phi^{(0)}_{tt}-c_0^2\Delta\phi^{(0)}+\alpha_pc_0^2T_t^{(0)}=0,&x\in\mb{R}^n,\ t>0,\\[0.26em]
		\displaystyle{T^{(0)}_t-\gamma D_{\Th}\Delta T^{(0)}+\frac{\gamma-1}{\alpha_p}\Delta\phi^{(0)}=0},&x\in\mb{R}^n,\ t>0,\\
		\phi^{(0)}(0,x)=\phi_0(x),\ \phi_t^{(0)}(0,x)=\phi_1(x),\ T^{(0)}(0,x)=T_0(x),&x\in\mb{R}^n,
	\end{cases}
\end{align}
which exactly coincides with the thermoelastic acoustic systems in isotropic solids (see \cite[Equations (34a)-(34b)]{Karlsen-Bruus=2015}). However, the rigorous justification of the global (in time) inviscid limit is non-trivial because the viscous dissipation $-(1+\beta)\nu_0\Delta\phi_t$ disappears. As a preparation, we will estimate the solution $\phi^{(0)}=\phi^{(0)}(t,x)$ in Subsection \ref{Sub-Section-Elastic-Acoustic}. Then, with suitable energies in the phase space, we demonstrate the following convergence results:
\begin{align*}
(\phi_t,\Delta \phi,T)\to (\phi_t^{(0)},\Delta \phi^{(0)},T^{(0)})\ \ \mbox{and}\ \ \phi\to\phi^{(0)}\ \  \mbox{in}\ \ L^{\infty}([0,\infty)\times\mb{R}^n)
\end{align*}
when $\nu_0\downarrow0$ with the rate $\sqrt{\nu_0}$. Remark that the convergence rate $\sqrt{\nu_0}$ admits to the formal WKB expansion of solutions to the Cauchy problem \eqref{Thermo-Acoustic-System} by the multi-scale analysis in Appendix \ref{Appendix-WKB-expansions}.

\medskip
\noindent\textbf{Notation:} Let us  denote the identity matrix by $I_{3\times 3}:=\diag(1,1,1)$, and the diagonal matrix with exponential elements by
\begin{align*}
	\diag\big(\mathrm{e}^{\lambda_j(|\xi|)t}\big)_{j=1}^3:=\diag\big(\mathrm{e}^{\lambda_1(|\xi|)t},\mathrm{e}^{\lambda_2(|\xi|)t},\mathrm{e}^{\lambda_3(|\xi|)t}\big).
\end{align*}
 We define the following zones in the phase space:
\begin{align*}
	\ml{Z}_{\intt}(\varepsilon_0):=\{|\xi|\leqslant\varepsilon_0\ll1\},\ \ 
	\ml{Z}_{\bdd}(\varepsilon_0,N_0):=\{\varepsilon_0\leqslant |\xi|\leqslant N_0\},\ \ 
	\ml{Z}_{\extt}(N_0):=\{|\xi|\geqslant N_0\gg1\}.
\end{align*}
The cut-off functions $\chi_{\intt}(\xi),\chi_{\bdd}(\xi),\chi_{\extt}(\xi)\in \mathcal{C}^{\infty}$ with supports in their corresponding zones $\ml{Z}_{\intt}(\varepsilon_0)$, $\ml{Z}_{\bdd}(\varepsilon_0/2,2N_0)$ and $\ml{Z}_{\extt}(N_0)$, respectively, so that $\chi_{\bdd}(\xi)=1-\chi_{\extt}(\xi)-\chi_{\intt}(\xi)$ for any $\xi \in \mb{R}^n$. The inequality $f\lesssim g$ means that there exists a positive constant $C$ satisfying $f\leqslant Cg$, which may be changed from line to line, analogously, for the inequality $f\gtrsim g$. The asymptotic relation  $f\simeq  g$ holds if and only if $g\lesssim f\lesssim g$, which is always used for optimal estimates.  We take the notation $\circ$ as the inner product in Euclidean space. The symbols of pseudo-differential operators $|D|^s$ and $\langle D\rangle^s $ with $s\in\mb{R}$ are denoted by $|\xi|^s$ and $\langle \xi\rangle^s$, respectively, with the Japanese bracket $\langle \xi\rangle^2:=1+|\xi|^2$. Let us recall the weighted $L^1$ space via
\begin{align*}
	L^{1,1}:=\left\{f\in L^1 \ \big|\ \|f\|_{L^{1,1}}:=\int_{\mb{R}^n}(1+|x|)|f(x)|\mathrm{d}x<\infty \right\}.
\end{align*}
 The (weighted) means of a summable function $f$ are denoted by
\begin{align*}
	\mb{R}\ni P_f:=\int_{\mb{R}^n}f(x)\mathrm{d}x	\ \ \mbox{and}\ \ \mb{R}^n\ni M_f:=\int_{\mb{R}^n}xf(x)\mathrm{d}x.
\end{align*}
The time-dependent function $\ml{A}_n(t)$ to be the growth rates ($n=1,2$) or decay rates ($n\geqslant 3$) has been defined in \eqref{A(t)}. Finally, we take the following two notations:
\begin{align}\label{G0G1}
	\Gamma_0:=\frac{(1+\beta)\nu_0+(\gamma-1)D_{\Th}}{2}\ \ \mbox{and}\ \ \Gamma_1:=\frac{-\Gamma_0^2-2D_{\Th}\Gamma_0+(1+\beta)\nu_0\gamma D_{\Th}}{2\sqrt{\gamma}c_0}.
\end{align}

\section{Main results}\label{Section_Main_Result}
\subsection{Large time behavior for the thermoviscous acoustic systems}
$\ \ \ \ $Our first result shows well-posedness, optimal decay estimates and large time asymptotic profiles of the  energy term $\Phi=\Phi(t,x)$ defined in \eqref{Micro-energy}, which consists of derivatives of the acoustic velocity potential $\phi_t$, $|D|\phi$, and the temperature $T$. Note that $\Phi_0(x):=\Phi(0,x)$ denotes the initial value of the energy term. Moreover, let us introduce a constant matrix $U_1$ such that
\begin{align}\label{U_1}
	U_1:=\begin{pmatrix}
		0&0&-\frac{2\gamma \alpha_p c_0^2}{\gamma-1}\\[0.26em]
		0&-\frac{2\gamma \alpha_p c_0^2}{\gamma-1}&0\\[0.26em]
		1&1&1
	\end{pmatrix},
\end{align}
whose construction will be shown at the beginning of Subsection \ref{Sub-sec-diag-small}.
\begin{theorem}\label{Thm-Diag}
Suppose that initial data $\Phi_0\in H^s$ with $s\geqslant0$. Then, there is a unique determined Sobolev solution  in the sense of an energy   to the Cauchy problem \eqref{Thermo-Acoustic-System} such that $\Phi\in\ml{C}([0,\infty),H^s)$. By assuming an additional $L^1$ integrability of initial data $\Phi_0\in H^s\cap L^1$ with $s\geqslant0$, the following optimal decay estimates hold:
\begin{align}\label{Est-Decay-Est-01}
	t^{-\frac{n+2s}{4}}|P_{\Phi_0}|\lesssim \|\Phi(t,\cdot)\|_{\dot{H}^{s}}\lesssim t^{-\frac{n+2s}{4}}\|\Phi_0\|_{H^s\cap L^1}
\end{align}
for large time $t\gg1$, provided that $|P_{\Phi_0}|\neq0$. Namely, the optimal decay estimates $\|\Phi(t,\cdot)\|_{\dot{H}^s}\simeq t^{-\frac{n+2s}{4}}$ hold for $n\geqslant 1$ and any $t\gg1$. Furthermore, the following refined estimates hold:
\begin{align}\label{Est-Decay-Est-02}
\|\Phi(t,\cdot)-U_1\Psi(t,\cdot)\|_{\dot{H}^{s}}\lesssim(1+t)^{-\frac{n+2s}{4}-\frac{1}{2}}\|\Phi_0\|_{H^s\cap L^1},
\end{align}
where $\Psi=\Psi(t,x)$ is the solution to the reference system \eqref{Ref-System} with its initial data $\Psi_0(x):=U_1^{-1}\Phi_0(x)$. 
\end{theorem}
\begin{remark}
In comparison with the error estimates \eqref{Est-Decay-Est-02}, the optimal decay rate from \eqref{Est-Decay-Est-01} can be improved $t^{-\frac{1}{2}}$ by subtracting the function $U_1\Psi(t,\cdot)$ in the $\dot{H}^s$ norm. Thus, we may explain the large time profile of the energy term $\Phi$ by the diffusion waves function $\Psi$.
\end{remark}
\begin{remark}
Motivated by the profile $\Psi$ satisfying \eqref{Ref-System}, the optimal decay property is generated by the combined influence of the viscous effect with $(1+\beta)\nu_0>0$ and the diffusion effect with $(\gamma-1)D_{\Th}>0$ due to the nature of the parameter $\Gamma_0>0$.
\end{remark}

Since the energy term $\Phi$ does not contain the solution itself $\phi$, our second result turns to the  optimal $L^2$ estimates of the velocity potential $\phi$, where it grows polynomially when $n = 1$ and logarithmically when $n = 2$, but decay polynomially when $n\geqslant 3$. 
\begin{theorem}\label{Thm-Optimal-Growth}
Suppose that initial data $\phi_0,\phi_1,T_0\in L^2\cap L^1$. Then, the solution to the Cauchy problem \eqref{Thermo-Acoustic-System} satisfies the following optimal estimates:
\begin{align}\label{Est-phi-Est}
    \ml{A}_n(t)|P_{\phi_1}|\lesssim\|\phi(t,\cdot)\|_{L^2}\lesssim \ml{A}_n(t)\|(\phi_0,\phi_1,T_0)\|_{(L^2\cap L^1)^3},
\end{align}
for large time $t\gg1$, provided that $|P_{\phi_1}|\neq0$, where the time-dependent coefficient $\ml{A}_n(t)$ is defined in \eqref{A(t)}. Namely, the optimal estimates $\|\phi(t,\cdot)\|_{L^2}\simeq \ml{A}_n(t)$ hold for $n\geqslant 1$ and any $t\gg1$.
\end{theorem}

\begin{remark}\label{Rem-2.3}
Comparing the last theorem with the optimal decay result in Theorem \ref{Thm-Diag}, we find that the velocity potential and the energy term have different properties, e.g. the decay property for the energy term $\Phi$ and the growth property for the velocity potential $\phi$ when $n=1,2$. In other words, even though the combined influence of viscous effect and thermal effect in the acoustic systems \eqref{Thermo-Acoustic-System}, the velocity potential still grows polynomially or logarithmically for lower dimensions. It was coursed by the wave structure $(\partial_t^2-c_0^2\Delta)\phi$ in the first equation of the thermoviscous acoustic systems \eqref{Thermo-Acoustic-System}.
\end{remark}

\begin{remark}\label{Rem-Wave-structure}
Let us recall the well-known free wave equation
\begin{align}\label{Eq-Wave}
	\begin{cases}
		w_{tt}-c_0^2\Delta w=0,&x\in\mb{R}^n,\ t>0,\\
		w(0,x)=w_0(x), \ w_t(0,x)=w_1(x),&x\in\mb{R}^n,
	\end{cases}
\end{align}
 for $n=1,2$. Taking $L^2\cap L^1$ integrability with $|P_{w_1}|\neq0$, the recent works \cite{Ikehata=2022-wave,Chen-Takeda=2023} got the following optimal growth estimates:
\begin{align*}
	\|w(t,\cdot)\|_{L^2}\simeq\begin{cases}
		\sqrt{t}&\mbox{when}\ \ n=1,\\
		\sqrt{\ln t}&\mbox{when}\ \ n=2,
	\end{cases}
\end{align*}
for large time $t\gg1$. From the optimal estimates \eqref{Est-phi-Est} in Theorem \ref{Thm-Optimal-Growth},  we notice that the growth rates for the free wave equation \eqref{Eq-Wave} and the thermoviscous acoustic systems \eqref{Thermo-Acoustic-System} are exactly the same when $n=1,2$, but the velocity potential decays polynomially with the aid of thermal-viscous effects when $n\geqslant 3$. One may see Table \ref{tab:table1} in detail.
\renewcommand\arraystretch{1.4}
\begin{table}[h!]
	\begin{center}
		\caption{Influence from the wave effect and thermal-viscous effects}
		\medskip
		\label{tab:table1}
		\begin{tabular}{cccc} 
			\toprule
			Dimensions & $n=1$ & $n=2$ & $n\geqslant3$\\
			\midrule
			Free wave property & $\sqrt{t}$ & $\sqrt{\log t}$ & -- \\
			Heat property (thermal-viscous effects) & $t^{-\frac{1}{4}}$& $t^{-\frac{1}{2}}$& $t^{-\frac{n}{4}}$\\  
			Thermoviscous acoustic systems property & $\sqrt{t}$ & $\sqrt{\log t}$ & $t^{\frac{1}{2}-\frac{n}{4}}=t^{\frac{1}{2}}\cdot t^{-\frac{n}{4}}$\\
			\hline 
			\multirow{2}{*}{Crucial influence} & \multirow{2}{*}{Wave effect} & \multirow{2}{*}{Wave effect} & Wave effect \\
			& & & \& Thermal-viscous effects \\
			\bottomrule
			\multicolumn{4}{l}{\emph{$*$ The terminology ``property'' specializes the time-dependent coefficients in the $L^2$ estimates of the solution.}}
		\end{tabular}
	\end{center}
\end{table}

\noindent It is worth noting that all large time properties in Table \ref{tab:table1} are optimal in the $L^2$ norm. Among them, concerning $n=1,2$, we realize that the large time property for the  thermoviscous acoustic systems \eqref{Thermo-Acoustic-System} does not influenced neither by the thermal effect nor the viscous effect. This is one of our new discoveries.
\end{remark}

Let us turn to the study for asymptotic profiles for large time, particularly, the optimal leading term. Before doing this, let us recall the first asymptotic profile in \eqref{first-leading-term} such that
\begin{align*}
\varphi(t,x):=\ml{G}_0(t,x)P_{\phi_1}\ \ \mbox{with}\ \ \ml{G}_0(t,x):=\ml{F}^{-1}_{\xi\to x}\left[\frac{\sin(\sqrt{\gamma}c_0|\xi|t)}{\sqrt{\gamma}c_0|\xi|}\mathrm{e}^{-\Gamma_0|\xi|^2t}\right],
\end{align*}
and introduce the second asymptotic profile as follows:
\begin{align}\label{psi-function}
    \psi(t,x):=-\nabla\ml{G}_0(t,x)\circ M_{\phi_1}+\ml{G}_1(t,x)P_{\phi_0}+\big(\ml{H}_0(t,x)+\ml{G}_2(t,x)\big)P_{\phi_1}+\ml{G}_3(t,x)P_{T_0},
\end{align}
where the auxiliary functions are defined by
\begin{align*}
    \ml{G}_1(t,x)&:=\ml{F}^{-1}_{\xi\to x}\left[\cos\left(\sqrt{\gamma}c_0|\xi|t\right)\mathrm{e}^{-\Gamma_0|\xi|^2t}\right],\\
    \ml{G}_2(t,x)&:=\ml{F}^{-1}_{\xi\to x}\left[\frac{(\gamma-1)D_{\Th}}{\gamma c_0^2}\left(\mathrm{e}^{-D_{\Th}|\xi|^2t}-\cos\left(\sqrt{\gamma}c_0|\xi|t\right)\mathrm{e}^{-\Gamma_0|\xi|^2t}\right)\right],\\
    \ml{G}_3(t,x)&:=\ml{F}^{-1}_{\xi\to x}\left[\alpha_p D_{\Th}\left(\mathrm{e}^{-D_{\Th}|\xi|^2t}-\cos\left(\sqrt{\gamma}c_0|\xi|t\right)\mathrm{e}^{-\Gamma_0|\xi|^2t}\right)\right],
  \end{align*}
and
\begin{align*}
    \ml{H}_0(t,x)&:=-\frac{\Gamma_1}{\sqrt{\gamma}c_0}t\Delta\ml{G}_1(t,x),
\end{align*}
with $\Gamma_0,\Gamma_1$ denoted in \eqref{G0G1}. We next state the optimal estimates for the error term by subtracting the leading term if $|M_{\phi_1}|\neq0$. This result implies that the dominant large time profile of the velocity potential for the thermoviscous acoustic systems \eqref{Thermo-Acoustic-System} is given by the diffusion waves function $\varphi$.
\begin{theorem}\label{Thm-Optimal-Leading}
Suppose that initial data $\phi_0,\phi_1,T_0\in L^2\cap L^1$ and $\phi_1\in L^{1,1}$ additionally. Then,  the solution to the Cauchy problem \eqref{Thermo-Acoustic-System} satisfies the following optimal error estimates:
\begin{align}\label{Est-phi-varphi-Est}
    t^{-\frac{n}{4}}|M_{\phi_1}|\lesssim\|\phi(t,\cdot)-\varphi(t,\cdot)\|_{L^2}\lesssim t^{-\frac{n}{4}}\|(\phi_0,\phi_1,T_0)\|_{(L^2\cap L^1)\times(L^2\cap L^{1,1})\times(L^2\cap L^1)}
\end{align}
for large time $t\gg1$, provided that $|M_{\phi_1}|\neq0$.
Namely, the optimal error estimates $\|\phi(t,\cdot)-\varphi(t,\cdot)\|_{L^2}\simeq t^{-\frac{n}{4}}$ hold for $n\geqslant 1$ and any $t\gg1$.
Furthermore, the solution fulfills the following refined error estimates:
\begin{align}\label{Est-phi-varphi-psi-Est}
    \|\phi(t,\cdot)-\varphi(t,\cdot)-\psi(t,\cdot)\|_{L^2}= o(t^{-\frac{n}{4}})
\end{align}
for large time $t\gg1$,  whose right-hand side depends on some norms of initial data.
\end{theorem}
\begin{remark}
    Comparing with the optimal estimates \eqref{Est-phi-Est}, by subtracting the function $\varphi(t,\cdot)$ in the $L^2$ norm, we obtain the  refined estimates \eqref{Est-phi-varphi-Est} with some improvements on decay rates for large time. Namely, the function $\varphi$ is the optimal leading term of the velocity potential $\phi$. Furthermore,  by subtracting  the additional function $\psi(t,\cdot)$ in the $L^2$ norm, we get faster decay estimates \eqref{Est-phi-varphi-psi-Est} for large time. In other words, $\psi$ is the second asymptotic profile of the velocity potential $\phi$.
\end{remark}
\begin{remark}\label{Rem-Interplay}
Let us recall the positive constant $\Gamma_0$ defined in \eqref{G0G1}. The optimal leading term $\varphi$ for the thermoviscous acoustic systems \eqref{Thermo-Acoustic-System} shows the interplay among 
\begin{itemize}
	\item the acoustic wave part $\frac{\sin(\sqrt{\gamma}c_0|\xi|t)}{\sqrt{\gamma}c_0|\xi|}$ with the speed of sound $c_0>0$;
	\item the thermal part $\mathrm{e}^{-\frac{\gamma-1}{2}D_{\Th}|\xi|^2t}$ with the thermal diffusion constant $D_{\Th}>0$;
	\item the viscous part $\mathrm{e}^{-\frac{1+\beta}{2}\nu_0|\xi|^2t}$ with the momentum diffusion constant $\nu_0>0$.
\end{itemize}
\end{remark}
\begin{remark}
The assumption $\phi_0,\phi_1,T_0\in L^2\cap L^1$ is not sufficient to guarantee large time stabilities of the velocity potential $\phi(t,\cdot)$ in the $L^2$ norm for lower dimensions $n=1,2$ (see Theorem \ref{Thm-Optimal-Growth} and Remark \ref{Rem-2.3} in detail). By taking $\phi_1\in L^{1,1}$ and $P_{\phi_1}=0$ additionally so that $\varphi(t,x)\equiv0$ in Theorem \ref{Thm-Optimal-Leading}, we notice that $\|\phi(t,\cdot)\|_{L^2}\simeq t^{-\frac{n}{4}}$ if $|M_{\phi_1}|\neq0$, which provides polynomial $L^2$ stabilities for all $n\geqslant 1$.
\end{remark}

\subsection{Inviscid limits for the thermoviscous acoustic systems}
$\ \ \ \ $Our next result focuses on limiting processes from the thermoviscous acoustic systems \eqref{Thermo-Acoustic-System} to the thermoelastic acoustic systems \eqref{Thermo-Acoustic-System-zero} as the momentum diffusion constant $\nu_0\downarrow 0$. We are interested in global (in time) convergences in the  $L^{\infty}$ framework for energy terms as well as the acoustic velocity potential.
\begin{theorem}\label{Thm-Inviscid-Energy}
Suppose that initial data $\langle D\rangle^{s+3}\phi_0,\langle D\rangle^{s+2}\phi_1,\langle D\rangle^{s+2}T_0\in L^1$ with $s>n+2$ for $n\geqslant 2$. Then, the following estimates hold:
\begin{align*}
    \sup_{t\in\left[0,\infty\right)}\left\|\left(\phi_t-\phi_t^{(0)},\Delta\phi-\Delta\phi^{(0)},T-T^{(0)}\right)(t,\cdot)\right\|_{L^{\infty}}&\leqslant C\sqrt{\nu_0}\|\langle D\rangle ^{s+2}(\langle D\rangle\phi_0,\phi_1,T_0)\|_{(L^1)^3},\\
    \sup_{t\in\left[0,\infty\right)}\left\|(\phi-\phi^{(0)})(t,\cdot)\right\|_{L^\infty}&\leqslant C\sqrt{\nu_0}\|\langle D\rangle ^{s}(\langle D\rangle\phi_0,\phi_1,T_0)\|_{(L^1)^3},
\end{align*}
where $C$ is a positive constant independent of $\nu_0$.
\end{theorem}
\begin{remark}
    The previous theorem shows that 
    \begin{align*}
    (\phi_t,\Delta \phi,T)\to (\phi_t^{(0)},\Delta \phi^{(0)},T^{(0)})\ \ \mbox{and}\ \ \phi\to\phi^{(0)}\ \  \mbox{in}\ \ L^{\infty}([0,\infty)\times\mb{R}^n) \ \   \mbox{as}\ \ \nu_0\downarrow 0
    \end{align*}
    with the rate of convergences $\sqrt{\nu_0}$ for $n\geqslant 2$. From our proofs (see \eqref{u-t} and \eqref{u-D2} in detail), we can find that $\phi_t-\phi_t^{(0)}$ and $\Delta\phi-\Delta\phi^{(0)}$ satisfy the above convergences when $n=1$. With additional assumptions  $\phi_1\in L^{1,1}$ with $|P_{\phi_1}|=0$, we can also obtain the above convergence for $T-T^{(0)}$ when $n=1$.
\end{remark}
\begin{remark}
    By using our energy method proposed in this paper associated with the Fourier analysis in the $L^2$ framework,  global (in time) convergence results in $L^{\infty}([0,\infty),L^2)$ also can be obtained as $\nu_0\downarrow0$ under different assumptions on initial data.
\end{remark}

\section{Large time behavior for the  energy term}\label{Section-Decay-Energy}
$\ \ \ \ $This section contributes to qualitative properties of the  energy term $\Phi=\Phi(t,x)$ defined in \eqref{Micro-energy} via diagonalization procedure, which gives the completed proof of Theorem \ref{Thm-Diag}. First of all, according to straightforward computations associated with \eqref{Thermo-Acoustic-System}$_2$, we may replace \eqref{Thermo-Acoustic-System}$_1$ by the next equation:
\begin{align}\label{Thermo-Acoustic-System-modified}
0&=\phi_{tt}-c_0^2\Delta\phi-(1+\beta)\nu_0\Delta \phi_t+\alpha_p c_0^2\left(\gamma D_{\Th}\Delta T-\frac{\gamma-1}{\alpha_p}\Delta \phi \right)\notag
\\
&=\phi_{tt}-\gamma c_0^2\Delta \phi -(1+\beta)\nu_0\Delta\phi_t+\alpha_p c_0^2\gamma D_{\Th}\Delta T.
\end{align}
It fulfills the following coupled system:
\begin{align}\label{Phi-Equ}
\begin{cases}
\Phi_t+A_1\sqrt{-\Delta}\Phi-A_2\Delta \Phi=0,&x\in\mb{R}^n,\ t>0,\\
\Phi(0,x)=\Phi_0(x),&x\in\mb{R}^n,
\end{cases}
\end{align}
where the coefficient matrices are
\begin{align*}
A_1:=\left(\begin{array}{ccc}
	-i\sqrt{\gamma}c_0 & 0 & 0\\[0.26em]
	0 & i\sqrt{\gamma}c_0 & 0\\[0.26em]  
	\frac{i(\gamma-1)}{2\alpha_p\sqrt{\gamma}c_0} & -\frac{i(\gamma-1)}{2\alpha_p\sqrt{\gamma}c_0} & 0
\end{array}\right)\ \ \mbox{and}\ \ A_2:=\left(\begin{array}{ccc}
\frac{(1+\beta)\nu_0}{2} & \frac{(1+\beta)\nu_0}{2} & -\alpha_pc_0^2\gamma D_{\Th} \\[0.26em]
\frac{(1+\beta)\nu_0}{2} & \frac{(1+\beta)\nu_0}{2} & -\alpha_pc_0^2\gamma D_{\Th}\\[0.26em]  
0 & 0 & \gamma D_{\Th}
\end{array}\right),
\end{align*}
with initial data $\Phi_0=\Phi_0(x)$ such that
\begin{align*}
	\Phi_0:=\big(\phi_1+i\sqrt{\gamma}c_0|D|\phi_0,\phi_1-i\sqrt{\gamma}c_0|D|\phi_0,T_0 \big)^{\mathrm{T}}.
\end{align*}
The thermoviscous acoustic systems \eqref{Thermo-Acoustic-System} do not belong to hyperbolic-parabolic coupled systems due to the higher order viscous term $-(1+\beta)\nu_0\Delta\phi_t$. For this reason, it seems that the classical results for treating hyperbolic-parabolic coupled systems do not work well in our models (e.g. the non-symmetric matrices $A_1$ and $A_2$ violate the condition in \cite[Lemma 2.2]{Umeda-Kawashima-Shizuta}). Later, motivated by diagonalization procedure \cite{Reissig-Wang=2005,Yang-Wang=2006,Jachmann=2008}, we will decouple the system \eqref{Phi-Equ} by the frequency analysis in the phase space to study sharp decay estimates and large time asymptotic profiles of the energy term.

Let us apply the partial Fourier transform with respect to spatial variables for the coupled system \eqref{Phi-Equ}, we may deduce
\begin{align}\label{Phi-Equ-Fourier}
	\begin{cases}
		\widehat{\Phi}_t+( A_1|\xi|+A_2|\xi|^2) \widehat{\Phi}=0,&\xi\in\mb{R}^n,\ t>0,\\
		\widehat{\Phi}(0,\xi)=\widehat{\Phi}_0(\xi),&\xi\in\mb{R}^n.
	\end{cases}
\end{align}
Our consideration for deriving asymptotic behavior of $\Phi$ in the phase space is divided into the following three cases in regard to the magnitude of frequencies.
\begin{itemize}
	\item Asymptotic representation of $\widehat{\Phi}$ when $\xi\in\ml{Z}_{\intt}(\varepsilon_0)$ via diagonalization procedure.
	\item Asymptotic representation of $\widehat{\Phi}$ when $\xi\in\ml{Z}_{\extt}(N_0)$ via diagonalization procedure.
	\item Exponential stability of $\widehat{\Phi}$ when $\xi\in\ml{Z}_{\bdd}(\varepsilon_0,N_0)$ via continuity and compactness.
\end{itemize}
\subsection{Diagonalization procedure for small frequencies}\label{Sub-sec-diag-small}
$\ \ \ \ $It is clear that the coefficient $A_1|\xi|$ exerts dominant influence comparing with another coefficient $A_2|\xi|^2$ for $\xi\in\ml{Z}_{\intt}(\varepsilon_0)$. For this reason, we may start diagonalization procedure with the matrix $A_1|\xi|$. 

Let us introduce a quantity  $\widehat{\Phi}^{(1,s)}=\widehat{\Phi}^{(1,s)}(t,\xi)$  such that
\begin{align*}
\widehat{\Phi}^{(1,s)}(t,\xi):=U_1^{-1}\widehat{\Phi}(t,\xi)
\end{align*}
  with the matrix $U_1$ defined in \eqref{U_1}. Multiplying the equation in \eqref{Phi-Equ-Fourier} by $U_1^{-1}$, then we obtain
\begin{align}\label{Eq_S3_01}
\widehat{\Phi}^{(1,s)}_t+|\xi|\Lambda_1^{(s)}\widehat{\Phi}^{(1,s)}+|\xi|^2A_2^{(1,s)}\widehat{\Phi}^{(1,s)}=0,
\end{align}
where the coefficient matrices are given by
\begin{align*}
	\Lambda_1^{(s)}&=U_1^{-1}A_1U_1=\diag(0,i\sqrt{\gamma}c_0,-i\sqrt{\gamma}c_0),\\
	A_2^{(1,s)}&=U_1^{-1}A_2U_1=\begin{pmatrix}
		D_{\Th}&D_{\Th}-(1+\beta)\nu_0&D_{\Th}-(1+\beta)\nu_0\\[0.26em]
		\frac{1}{2}(\gamma-1)D_{\Th}&\frac{1}{2}[(1+\beta)\nu_0+(\gamma-1)D_{\Th}]&\frac{1}{2}[(1+\beta)\nu_0+(\gamma-1)D_{\Th}]\\[0.26em]
		\frac{1}{2}(\gamma-1)D_{\Th}&\frac{1}{2}[(1+\beta)\nu_0+(\gamma-1)D_{\Th}]&\frac{1}{2}[(1+\beta)\nu_0+(\gamma-1)D_{\Th}]
	\end{pmatrix}.
\end{align*}

Next, in order to retain the derived diagonal matrix $|\xi|\Lambda_1^{(s)}$ as the dominant coefficient,  we may introduce an ansatz $\widehat{\Phi}^{(2,s)}=\widehat{\Phi}^{(2,s)}(t,\xi)$ such that
\begin{align*}
\widehat{\Phi}^{(2,s)}(t,\xi):=U_2^{-1}\widehat{\Phi}^{(1,s)}(t,\xi)	
\end{align*}
 with $U_2:=I_{3\times 3}+|\xi| N_2$ carrying
\begin{align}\label{N_2}
	N_2:=\begin{pmatrix}
		0&\frac{D_{\Th}-(1+\beta)\nu_0}{i\sqrt{\gamma}c_0}&-\frac{D_{\Th}-(1+\beta)\nu_0}{i\sqrt{\gamma}c_0}\\[0.26em]-\frac{(\gamma-1)D_{\Th}}{2i\sqrt{\gamma}c_0}&0&
		-\frac{(1+\beta)\nu_0+(\gamma-1)D_{\Th}}{4i\sqrt{\gamma}c_0}\\[0.26em]
		\frac{(\gamma-1)D_{\Th}}{2i\sqrt{\gamma}c_0}&
		\frac{(1+\beta)\nu_0+(\gamma-1)D_{\Th}}{4i\sqrt{\gamma}c_0}&0
	\end{pmatrix}.
\end{align}
By directly multiplying the matrix $U_2^{-1}$ on the left-hand side of \eqref{Eq_S3_01}, one notices
\begin{align*}
	\widehat{\Phi}^{(2,s)}_t+|\xi|U_2^{-1}\Lambda_1^{(s)}U_2\widehat{\Phi}^{(2,s)}+|\xi|^2U_2^{-1}A_2^{(1,s)}U_2\widehat{\Phi}^{(2,s)}=0.
\end{align*}
In the above equation, two coefficient matrices can be expanded by
\begin{align*}
	|\xi|U_2^{-1}\Lambda_1^{(s)}U_2&=|\xi|(I_{3\times 3}+|\xi| N_2)^{-1}\Lambda_1^{(s)}(I_{3\times 3}+|\xi| N_2)\\
	&=|\xi|(I_{3\times 3}-|\xi| N_2U_2^{-1})\Lambda_1^{(s)}(I_{3\times 3}+|\xi | N_2)\\
	&=|\xi|\Lambda_1^{(s)}-|\xi|^2N_2U_2^{-1}\Lambda_1^{(s)}U_2+|\xi|^2\Lambda_1^{(s)}N_2\\
	&=|\xi|\Lambda_1^{(s)}-|\xi|^2N_2(I_{3\times 3}-|\xi| N_2U_2^{-1})\Lambda_1^{(s)}(I_{3\times 3}+|\xi| N_2)+|\xi|^2\Lambda_1^{(s)}N_2\\
	&=|\xi|\Lambda_1^{(s)}-|\xi|^2N_2\Lambda_1^{(s)}+|\xi|^2\Lambda_1^{(s)}N_2+|\xi|^3N_2^2U_2^{-1}\Lambda_1^{(s)}U_2-|\xi|^3N_2\Lambda_1^{(s)}N_2
\end{align*}
as well as
\begin{align*}
|\xi|^2U_2^{-1}A_2^{(1,s)}U_2&=|\xi|^2(I_{3\times 3}+|\xi| N_2)^{-1}A_2^{(1,s)}(I_{3\times 3}+|\xi| N_2)\\
&=|\xi|^2(I_{3\times 3}-|\xi| N_2U_2^{-1})A_2^{(1,s)}(I_{3\times 3}+|\xi| N_2)\\
&=|\xi|^2A_2^{(1,s)}-|\xi|^3N_2U_2^{-1}A_2^{(1,s)}U_2+|\xi|^3A_2^{(1,s)}N_2.
\end{align*}
Therefore, we arrive at 
\begin{align*}
\widehat{\Phi}^{(2,s)}_t+|\xi|\Lambda_1^{(s)}\widehat{\Phi}^{(2,s)}+|\xi|^2\left(A_2^{(1,s)}-[N_2,\Lambda_1^{(s)}]\right)\widehat{\Phi}^{(2,s)}+R_2^{(s)}\widehat{\Phi}^{(2,s)}=0,
\end{align*}
where $[N_2,\Lambda_1^{(s)}]:=N_2\Lambda_1^{(s)}-\Lambda_1^{(s)}N_2$ is the commutator, and the reminder is
\begin{align*}
R_2^{(s)}=\left(N_2^2U_2^{-1}\Lambda_1^{(s)}U_2-N_2\Lambda_1^{(s)}N_2-N_2U_2^{-1}A_2^{(1,s)}U_2+A_2^{(1,s)}N_2\right)|\xi|^3=\ml{O}(|\xi|^3).
\end{align*}
From the suitable matrix $N_2$, thanks to pairwise distinct elements in $\Lambda_1^{(s)}$, we immediately get
\begin{align*}
\Lambda_2^{(s)}=A_2^{(1,s)}-[N_2,\Lambda_1^{(s)}]=\diag(D_{\Th},\Gamma_0,\Gamma_0).
\end{align*}
That is to say,
\begin{align*}
\widehat{\Phi}^{(2,s)}_t+\left(|\xi|\Lambda_1^{(s)}+|\xi|^2\Lambda_2^{(s)}\right)\widehat{\Phi}^{(2,s)}+R_2^{(s)}\widehat{\Phi}^{(2,s)}=0.
\end{align*}

From the previous diagonalization procedure, we have obtained pairwise distinct characteristic roots from $|\xi|\Lambda_1^{(s)}+|\xi|^2\Lambda_2^{(s)}$ whose real parts are positive. Hence, following the general philosophy stated in \cite[Chapter 2]{Jachmann=2008}, we may derive asymptotic representations of the  energy term for small frequencies.
\begin{prop}\label{Prop-small-frenquencies}
	The characteristic roots $\lambda_j=\lambda_j(|\xi|)$ of the coefficient matrix $A_1|\xi|+A_2|\xi|^2$ in the equation of \eqref{Phi-Equ-Fourier} behave for $\xi\in\ml{Z}_{\intt}(\varepsilon_0)$ as
	\begin{align*}
		\lambda_1(|\xi|)&=-D_{\Th}|\xi|^2+\ml{O}(|\xi|^3),\\
		\lambda_{2,3}(|\xi|)&=\mp i\sqrt{\gamma}c_0|\xi|-\Gamma_0|\xi|^2+\ml{O}(|\xi|^3).
	\end{align*}
	 The solution to the Cauchy problem \eqref{Phi-Equ-Fourier} has the representation for $\xi\in\ml{Z}_{\intt}(\varepsilon_0)$ as follows: 
	\begin{align*}   
		\chi_{\intt}(\xi)\widehat{\Phi}(t,\xi)=\chi_{\intt}(\xi)U_{\intt}(|\xi|)\diag\big(\mathrm{e}^{\lambda_j(|\xi|)t}\big)_{j=1}^3 U_{\intt}^{-1}(|\xi|)\widehat{\Phi}_0(\xi),
	\end{align*}
	with $U_{\intt}(|\xi|):=U_1(I_{3\times3}+|\xi|N_2)$, where the matrices $U_1$ and $N_2$ are defined in \eqref{U_1} and \eqref{N_2}, respectively.
\end{prop}

\subsection{Diagonalization procedure for large frequencies}
$\ \ \ \ $Let us turn to the situation for large frequencies $\xi\in\ml{Z}_{\extt}(N_0)$, where the coefficient $A_2|\xi|^2$ plays a dominant role in comparison with another one $A_1|\xi|$. Thus, we may start  diagonalization procedure with $A_2|\xi|^2$. Recalling $(1+\beta)\nu_0\neq\gamma D_{\Th}$, we first introduce a quantity $\widehat{\Phi}^{(1,l)}=\widehat{\Phi}^{(1,l)}(t,\xi)$ such that
\begin{align}\label{Q_1}
\widehat{\Phi}^{(1,l)}(t,\xi):=Q_1^{-1}\widehat{\Phi}(t,\xi)\ \ \mbox{with}\ \ 	Q_1:=\begin{pmatrix}-1&\frac{\alpha_pc_0^2\gamma D_{\Th}}{(1+\beta)\nu_0-\gamma D_{\Th}}&1\\[0.26em]1&\frac{\alpha_pc_0^2\gamma D_{\Th}}{(1+\beta)\nu_0-\gamma D_{\Th}}&1\\[0.26em]0&1&0\end{pmatrix}.
\end{align}
This new function fulfills the following system:
\begin{align*}
	\widehat{\Phi}_t^{(1,l)}+|\xi|^2\Lambda_1^{(l)}\widehat{\Phi}^{(1,l)}+|\xi|A_1^{(1,l)}\widehat{\Phi}^{(1,l)}=0,
\end{align*}
in which
\begin{align*}
		\Lambda_1^{(l)}&=Q_1^{-1}A_2Q_1=\diag\big(0,\gamma D_{\Th},(1+\beta)\nu_0\big),\\
A_1^{(1,l)}&=Q_1^{-1}A_1Q_1=\begin{pmatrix}0&\frac{i\gamma\sqrt{\gamma}c_0^3\alpha_pD_{\Th}}{(1+\beta)\nu_0-\gamma D_{\Th}}&i\sqrt{\gamma}c_0\\[0.26em]-\frac{i(\gamma-1)}{\alpha_p\sqrt{\gamma}c_0}&0&0\\[0.26em] i\sqrt{\gamma}c_0\big(1+\frac{(\gamma-1)D_{\Th}}{(1+\beta)\nu_0-\gamma D_{\Th}}\big)&0&0\end{pmatrix}.
\end{align*}

Similarly to the case for small frequencies, we introduce a  quantity $\widehat{\Phi}^{(2,l)}=\widehat{\Phi}^{(2,l)}(t,\xi)$ such that
\begin{align*}
	\widehat{\Phi}^{(2,l)}(t,\xi):=Q_2^{-1}\widehat{\Phi}^{(1,l)}(t,\xi)
\end{align*}
 with $Q_2:=I_{3\times 3}+|\xi|^{-1}V_2$ as well as
 \begin{align}\label{V_2}
 	V_2&:=\begin{pmatrix}0&\frac{i\sqrt{\gamma}c_0^3\alpha_p}{(1+\beta)\nu_0-\gamma D_{\Th}}&\frac{i\sqrt{\gamma}c_0}{(1+\beta)\nu_0}\\[0.26em]\frac{i(\gamma-1)}{\alpha_p \gamma\sqrt{\gamma} c_0 D_{\Th}}&0&0\\[0.26em]-\frac{i\sqrt{\gamma}c_0}{(1+\beta)\nu_0}\big(1+\frac{(\gamma-1)D_{\Th}}{(1+\beta)\nu_0-\gamma D_{\Th}}\big)&0&0\end{pmatrix}.
 \end{align}
 This unknown satisfies
\begin{align*}
	\widehat{\Phi}_t^{(2,l)}+|\xi|^2Q_2^{-1}\Lambda_1^{(l)}Q_2\widehat{\Phi}^{(2,l)}+|\xi|Q_2^{-1}A_{1}^{(1,l)}Q_2\widehat{\Phi}^{(2,l)}=0,
\end{align*}
where two coefficient matrices can be expanded by
\begin{align*}
	Q_2^{-1}\Lambda_1^{(l)}Q_2&=\Lambda_1^{(l)}Q_2-|\xi|^{-1}V_2Q_2^{-1}\Lambda_1^{(l)}Q_2\\
	&=\Lambda_1^{(l)}+|\xi|^{-1}\Lambda_1^{(l)}V_2-|\xi|^{-1}V_2\Lambda_1^{(l)}Q_2+|\xi|^{-2}V_2^2Q_2^{-1}\Lambda_1^{(l)}Q_2\\
	&=\Lambda_1^{(l)}-|\xi|^{-1}[V_2,\Lambda_1^{(l)}]-|\xi|^{-2}V_2\Lambda_1^{(l)}V_2+|\xi|^{-2}V_2^2Q_2^{-1}\Lambda_1^{(l)}Q_2,
\end{align*}
and
\begin{align*}
	Q_2^{-1}A_1^{(1,l)}Q_2=A_1^{(1,l)}+|\xi|^{-1}A_1^{(1,l)}V_2-|\xi|^{-1}V_2A_1^{(1,l)}Q_2+|\xi|^{-2}V_2^2Q_2^{-1}A_1^{(1,l)}Q_2.
\end{align*}
It means 
\begin{align}\label{Eq_New-01}
	\widehat{\Phi}^{(2,l)}_t+|\xi|^2\Lambda_1^{(l)}\widehat{\Phi}^{(2,l)}+|\xi|\big(A_1^{(1,l)}-[V_2,\Lambda_1^{(l)}]\big)\widehat{\Phi}^{(2,l)}+R_2^{(l)}\widehat{\Phi}^{(2,l)}=0,
\end{align}
in which the remainder can be expressed by
\begin{align*}
	R_2^{(l)}&=-V_2\Lambda_1^{(l)}V_2+V_2^2Q_2^{-1}\Lambda_1^{(l)}Q_2+A_1^{(1,l)}V_2-V_2A_1^{(1,l)}Q_2+|\xi|^{-1}V_2^2Q_2^{-1}A_1^{(1,l)}Q_2\\
	&=-V_2\Lambda_1^{(l)}V_2+V_2^2\Lambda_1^{(l)}+|\xi|^{-1}V_2^2\Lambda_1^{(l)}V_2-|\xi|^{-1}V_2^3Q_2^{-1}\Lambda_1^{(l)}Q_2+A_1^{(1,l)}V_2-V_2A_1^{(1,l)}\\
	&\quad\,-|\xi|^{-1}V_2A_1^{(1,l)}V_2+|\xi|^{-1}V_2^2Q_2^{-1}A_1^{(1,l)}Q_2.
\end{align*}
Thanks to the choice of $V_2$ when we constructed the ansatz, one may notice
\begin{align*}
	\Lambda_2^{(l)}:=A_1^{(1,l)}-[V_2,\Lambda_1^{(l)}]=0,
\end{align*} which says the vanishing coefficient of $|\xi|$ in \eqref{Eq_New-01}. Nevertheless, due to zero of the first element in $\Lambda_1^{(\ell)}$, it is not sufficient for us to ensure any stabilities. We have to explore asymptotic expansions of solutions in the phase space in depth.

To separate different influences via the magnitude of frequencies, let us extract the $\ml{O}(1)$-terms in the remainder $R_2^{(l)}$, in other words,
\begin{align*}
	A_1^{(2,l)}&=-V_2\Lambda_1^{(l)}V_2+V_2^2\Lambda_1^{(l)}+A_1^{(1,l)}V_2-V_2A_1^{(1,l)}\\
	&=V_2\left([V_2,\Lambda_1^{(l)}]-A_1^{(1,l)}\right)+A_1^{(1,l)}V_2=A_1^{(1,l)}V_2.
\end{align*}
Direct calculations show
\begin{align*}
	A_1^{(2,l)}&=\begin{pmatrix}\frac{c_0^2}{(1+\beta)\nu_0}&0&0\\[0.26em]0&\frac{(\gamma-1)c_0^2}{(1+\beta)\nu_0-\gamma D_{\Th}}&\frac{\gamma-1}{\alpha_p (1+\beta)\nu_0}\\[0.26em]0&-\frac{\gamma c_0^4\alpha_p[(1+\beta)\nu_0-D_{\Th}]}{[(1+\beta)\nu_0-\gamma D_{\Th}]^2}&-\frac{\gamma c_0^2[(1+\beta)\nu_0-D_{\Th}]}{[(1+\beta)\nu_0-\gamma D_{\Th}](1+\beta)\nu_0}\end{pmatrix}.
\end{align*}
As a consequence, we will treat the following system:
\begin{align*}
	\widehat{\Phi}^{(2,l)}_t+|\xi|^2\Lambda_1^{(l)}\widehat{\Phi}^{(2,l)}+A_1^{(2,l)}\widehat{\Phi}^{(2,l)} +\big(R_2^{(l)}-A_1^{(2,l)}\big)\widehat{\Phi}^{(2,l)}=0,
\end{align*}
with the new remainder $R_2^{(l)}-A_1^{(2,l)}=\ml{O}(|\xi|^{-1})$. Taking $\widehat{\Phi}^{(3,l)}=\widehat{\Phi}^{(3,l)}(t,\xi)$ such that 
\begin{align*}
	\widehat{\Phi}^{(3,l)}(t,\xi):=Q_3^{-1}\widehat{\Phi}^{(2,l)}(t,\xi)
\end{align*}
with $Q_3:=I_{3\times 3}+|\xi|^{-2}V_3$ and 
\begin{align}\label{M3}
	V_3:=\begin{pmatrix}0&0&0\\[0.26em]0&0&\frac{\gamma-1}{\alpha_p (1+\beta)\nu_0[(1+\beta)\nu_0-\gamma D_{\Th}]}\\[0.26em]0&\frac{\gamma c_0^4\alpha_p[(1+\beta)\nu_0-D_{\Th}]}{[(1+\beta)\nu_0-\gamma D_{\Th}]^3}&0\end{pmatrix},
\end{align}
by applying the same approach as those of last diagonalization procedures, we may arrive at
\begin{align*}
	\widehat{\Phi}^{(3,l)}_t+\big(|\xi|^2\Lambda_1^{(l)}+\Lambda_3^{(l)}\big)\widehat{\Phi}^{(3,l)}+R_3^{(l)}\widehat{\Phi}^{(3,l)}=0.
\end{align*}
In the above differential equation, we have determined
\begin{align*}
	\Lambda_3^{(l)}=A_1^{(2,l)}-[V_3,\Lambda_1^{(l)}]=\diag\left(\frac{c_0^2}{(1+\beta)\nu_0},\frac{(\gamma-1)c_0^2}{(1+\beta)\nu_0-\gamma D_{\Th}},-\frac{\gamma c_0^2[(1+\beta)\nu_0-D_{\Th}]}{[(1+\beta)\nu_0-\gamma D_{\Th}](1+\beta)\nu_0}\right),
\end{align*}
as well as
\begin{align*}
R_3^{(l)}&=Q_3^{-1}\big(R_2^{(l)}-A_1^{(2,l)}\big)Q_3+|\xi|^{-2}\big(V_3^2Q_3^{-1}\Lambda_1^{(l)}Q_3-V_3\Lambda_1^{(l)}V_3+A_1^{(2,l)}V_3-V_3Q_3^{-1}A_1^{(2,l)}Q_3\big).
\end{align*}
Remark that the remainder behaves as $R_3^{(l)}=\ml{O}(|\xi|^{-2})$ for large frequencies.

Summarizing all steps of diagonalization procedure in the above, thanks to  $(1+\beta)\nu_0\neq \gamma D_{\Th}$, we have obtained pairwise distinct characteristic roots from $|\xi|^2\Lambda_1^{(l)}+\Lambda_3^{(l)}$ whose real parts are positive. According to the general theory developed in \cite[Chapter 2]{Jachmann=2008}, the following proposition illustrates asymptotic representations of the  energy term for large frequencies.
\begin{prop}\label{Prop-large-frenquencies}
		The characteristic roots $\lambda_j=\lambda_j(|\xi|)$ of the coefficient matrix $A_1|\xi|+A_2|\xi|^2$ in the equation of \eqref{Phi-Equ-Fourier} behave for $\xi\in\ml{Z}_{\extt}(N_0)$ as
	\begin{align*}
		\lambda_1(|\xi|)&=-\frac{c_0^2}{(1+\beta)\nu_0}+\ml{O}(|\xi|^{-1}),\\
		\lambda_2(|\xi|)&=-\gamma D_{\Th}|\xi|^2-\frac{(\gamma-1)c_0^2}{(1+\beta)\nu_0-\gamma D_{\Th}}+\ml{O}(|\xi|^{-1}),\\
		\lambda_3(|\xi|)&=-(1+\beta)\nu_0|\xi|^2+\frac{\gamma c_0^2[(1+\beta)\nu_0-D_{\Th}]}{[(1+\beta)\nu_0-\gamma D_{\Th}](1+\beta)\nu_0}+\ml{O}(|\xi|^{-1}).
	\end{align*}
The solution to the Cauchy problem \eqref{Phi-Equ-Fourier} has the representation for $\xi\in\ml{Z}_{\extt}(N_0)$ as follows: 
	\begin{align*}
		\chi_{\extt}(\xi)\widehat{\Phi}(t,\xi)=\chi_{\extt}(\xi)Q_{\extt}(|\xi|)\diag\big(\mathrm{e}^{\lambda_j(|\xi|)t}\big)_{j=1}^3Q_{\extt}^{-1}(|\xi|)\widehat{\Phi}_0(\xi),
	\end{align*}
	with $Q_{\extt}(|\xi|):=Q_1(I_{3\times3}+|\xi|^{-1}V_2)(I_{3\times3}+|\xi|^{-2}V_3)$, where the matrix $Q_1$ is defined in \eqref{Q_1}, and the other matrices $V_2,V_3$ are defined in \eqref{V_2} and \eqref{M3}, respectively.
\end{prop}

\subsection{Exponential stability for bounded frequencies}
$\ \ \ \ $Eventually, we will derive an exponential decay result for the  energy term when $\xi\in\ml{Z}_{\bdd}(\varepsilon_0,N_0)$  by applying a contradiction argument. Assume that there is a pure imaginary characteristic root $\lambda=ia$ with $a\in\mb{R}\backslash\{0\}$ of the coefficient matrix $A_1|\xi|+A_2|\xi|^2$ for $\xi\in\mathscr{Z}_{\bdd}(\varepsilon_0,N_0)$. The non-trivial real number $a$ has to fulfill the equation
\begin{align*}
	0&=\det(A_1|\xi|+A_2|\xi|^2-ia I_{3\times 3})\\&=\begin{pmatrix}-i\sqrt{\gamma}c_0|\xi|+\frac{1}{2}(1+\beta)\nu_0|\xi|^2-ia&\frac{1}{2}(1+\beta)\nu_0|\xi|^2&-\alpha_p c_0^2\gamma D_{\Th}|\xi|^2\\[0.26em]\frac{1}{2}(1+\beta)\nu_0|\xi|^2&i\sqrt{\gamma}c_0|\xi|+\frac{1}{2}(1+\beta)\nu_0|\xi|^2-ia&-\alpha_p c_0^2\gamma D_{\Th}|\xi|^2\\[0.26em] \frac{i(\gamma-1)}{2\alpha_p\sqrt{\gamma}c_0}|\xi|&-\frac{i(\gamma-1)}{2\alpha_p\sqrt{\gamma}c_0}|\xi|&\gamma D_{\Th}|\xi|^2-ia\end{pmatrix}\\
	&=i a^3-\big((1+\beta)\nu_0+\gamma D_{\Th}\big)|\xi|^2a^2-i\big(\gamma c_0^2|\xi|^2+(1+\beta)\nu_0\gamma D_{\Th}|\xi|^4\big)a+c_0^2\gamma D_{\Th}|\xi|^4.
\end{align*}
We collect the real and imaginary parts separately, which implies
\begin{align*}
\begin{cases}
	a\big[a^2-\big(\gamma c_0^2|\xi|^2+(1+\beta)\nu_0\gamma D_{\Th}|\xi|^4\big)\big]=0,\\
-\big((1+\beta)\nu_0+\gamma D_{\Th}\big)|\xi|^2a^2+c_0^2\gamma D_{\Th}|\xi|^4=0.
\end{cases}
\end{align*}
The combination of these algebraic equations for solving $a^2\neq0$ yields
\begin{align*}
	\frac{\gamma c_0^2[(1+\beta)\nu_0+(\gamma-1)D_{\Th}]}{(1+\beta)\nu_0+\gamma D_{\Th}}|\xi|^2+(1+\beta)\nu_0\gamma D_{\Th}|\xi|^4=0.
\end{align*}
It gives a contradiction since non-negativities of the coefficients for $|\xi|^2$ and $|\xi|^4$. In other words, the characteristic roots of $A_1|\xi|+A_2|\xi|^2$ cannot be pure imaginary numbers. According to the compactness of the bounded zone $\ml{Z}_{\bdd}(\varepsilon_0,N_0)$ as well as the continuity of $\lambda_j(|\xi|)$ with respect to $|\xi|$, we straightforwardly  claim $\Re\lambda_j(|\xi|)<0$ for $\xi\in\ml{Z}_{\bdd}(\varepsilon_0,N_0)$ and the next stable result.
\begin{prop}\label{Prop-bounded-frenquencies}
	The solution to the Cauchy problem \eqref{Phi-Equ-Fourier} fulfills an exponential decay estimate
	\begin{align*}
		\chi_{\bdd}(\xi)|\widehat{\Phi}(t,\xi)|\lesssim \chi_{\bdd}(\xi)\mathrm{e}^{-ct}|\widehat{\Phi}_0(\xi)|
	\end{align*}
	for $t\geqslant0$ with a positive constant $c$.
\end{prop}

\subsection{Decay properties and asymptotic profiles: Proof of Theorem \ref{Thm-Diag} }\label{proof-of-thm-diag}
$\ \ \ \ $After deriving asymptotic representations of the  energy term, with the aid of  Propositions \ref{Prop-small-frenquencies}-\ref{Prop-bounded-frenquencies}, we get the following pointwise estimates in the phase space:
\begin{align}\label{pointwise-est-Phi}
	|\widehat{\Phi}(t,\xi)|\lesssim \mathrm{e}^{-c\frac{|\xi|^2}{1+|\xi|^2}t}|\widehat{\Phi}_0(\xi)|
\end{align}
	for $t\geqslant0$ with a positive constant $c$. It leads to the stability such that
	\begin{align*}
	\|\Phi(t,\cdot)\|_{H^s}=\|\langle \xi\rangle^s\widehat{\Phi}(t,\xi)\|_{L^2}\lesssim\|\langle\xi\rangle^s\widehat{\Phi}_0(\xi)\|_{L^2}=\|\Phi_0\|_{H^s}.
	\end{align*}
	 By combining the representations derived in Propositions \ref{Prop-small-frenquencies} and \ref{Prop-large-frenquencies}, we can demonstrate the desired well-posedness for the energy term $\Phi$. In order to derive decay properties of it, we may employ the Plancherel theorem to find
\begin{align*}
	\|\Phi(t,\cdot)\|_{\dot{H}^{s}}&\lesssim\left\|\chi_{\intt}(\xi)|\xi|^{s}\mathrm{e}^{-c|\xi|^2t}\widehat{\Phi}_0(\xi)\right\|_{L^2}+\mathrm{e}^{-ct}\|\chi_{\bdd}(\xi)\widehat{\Phi}_0(\xi)\|_{L^2}+\mathrm{e}^{-ct}\|\chi_{\extt}(\xi)\langle\xi\rangle^{s}\widehat{\Phi}_0(\xi)\|_{L^2}\\
	&\lesssim(1+t)^{-\frac{n+2s}{4}}\|\Phi_0\|_{L^1}+\mathrm{e}^{-ct}\|\Phi_0\|_{H^{s}},
\end{align*}
where we applied the Hausdorff-Young inequality and
\begin{align*}
	\left\|\chi_{\intt}(\xi)|\xi|^{s}\mathrm{e}^{-c|\xi|^2t}\right\|_{L^2}^2\lesssim\int_0^{\varepsilon_0}r^{2s+n-1}\mathrm{e}^{-2cr^2t}\mathrm{d}r\lesssim (1+t)^{-\frac{n+2s}{2}}.
\end{align*}
This is the first part of the theorem.

Next, we will derive asymptotic profiles of the  energy term with some refined estimates. Let us introduce the $|\xi|$-dependent functions $\mu_j(|\xi|)$ with $j=1,2,3$, which are the principal parts of  $\lambda_j(|\xi|)$ for small frequencies (see Proposition \ref{Prop-small-frenquencies} in detail), namely,
\begin{align*}
	\mu_1(|\xi|)&:=-D_{\Th}|\xi|^2,\\
	\mu_{2,3}(|\xi|)&:=\mp i\sqrt{\gamma}c_0|\xi|-\Gamma_0|\xi|^2.
\end{align*}
Let us recall the reference system \eqref{Ref-System} with initial data $\Psi_0(x):=U_1^{-1}\Phi_0(x)$, where the matrix $U_1$ was defined in \eqref{U_1}, whose solution in the phase space is given by 
\begin{align*}
	\widehat{\Psi}(t,\xi)=\diag\big(\mathrm{e}^{\mu_j(|\xi|)t}\big)_{j=1}^3 U_1^{-1}\widehat{\Phi}_0(\xi).
\end{align*}
Using the equality
\begin{align*}
(I_{3\times 3}+|\xi|N_2)^{-1}=I_{3\times 3}-|\xi|N_2(I_{3\times 3}+|\xi|N_2)^{-1},
\end{align*}
we may decompose the representation in Proposition \ref{Prop-small-frenquencies} by
\begin{align*}
	\chi_{\intt}(\xi)\widehat{\Phi}(t,\xi)&=\chi_{\intt}(\xi)U_1(I_{3\times 3}+|\xi|N_2)\diag\big(\mathrm{e}^{\lambda_j(|\xi|)t}\big)_{j=1}^3(I_{3\times 3}+|\xi|N_2)^{-1}U_1^{-1}\widehat{\Phi}_0(\xi)\\
	&=\chi_{\intt}(\xi)U_1\diag\big(\mathrm{e}^{\lambda_j(|\xi|)t}\big)_{j=1}^3U_1^{-1}\widehat{\Phi}_0(\xi)\\
	&\quad+\chi_{\intt}(\xi)|\xi|U_1N_2\diag\big(\mathrm{e}^{\lambda_j(|\xi|)t}\big)_{j=1}^3(I_{3\times 3}+|\xi|N_2)^{-1}U_1^{-1}\widehat{\Phi}_0(\xi)\\
	&\quad-\chi_{\intt}(\xi)|\xi|U_1\diag\big(\mathrm{e}^{\lambda_j(|\xi|)t}\big)_{j=1}^3N_2(I_{3\times 3}+|\xi|N_2)^{-1}U_1^{-1}\widehat{\Phi}_0(\xi)\\
	&=:S_{\intt,1}(t,\xi)+S_{\intt,2}(t,\xi)+S_{\intt,3}(t,\xi).
\end{align*}
From the following integral formula carrying suitable $g_j(|\xi|)$:
\begin{align*}
	\mathrm{e}^{\mu_j(|\xi|)t-g_j(|\xi|)t}-\mathrm{e}^{\mu_j(|\xi|)t}=-g_j(|\xi|)t\,\mathrm{e}^{\mu_j(|\xi|)t}\int_0^1\mathrm{e}^{-g_j(|\xi|)t\tau}\mathrm{d}\tau,
\end{align*}
we explore the estimate
\begin{align*}
	|S_{\intt,1}(t,\xi)-\chi_{\intt}(\xi)U_1\widehat{\Psi}(t,\xi)|&\lesssim\chi_{\intt}(\xi)\sup\limits_{j=1,2,3}\big|\mathrm{e}^{\lambda_j(|\xi|)t}-\mathrm{e}^{\mu_j(|\xi|)t}\big|\,|\widehat{\Phi}_0(\xi)|\\
	&\lesssim \chi_{\intt}(\xi)(1+t)|\xi|^3\mathrm{e}^{-c|\xi|^2t}|\widehat{\Phi}_0(\xi)|,
\end{align*}
in which we set $g_j(|\xi|)=\ml{O}(|\xi|^3)$ from Proposition \ref{Prop-small-frenquencies}. Then, due to the additional factor $|\xi|$, we may have the error estimates
\begin{align*}
	&\big\||\xi|^{s}\big(S_{\intt,1}(t,\xi)-\chi_{\intt}(\xi)U_1\widehat{\Psi}(t,\xi)\big)\big\|_{L^2}+\big\||\xi|^{s}\big(S_{\intt,2}(t,\xi)+S_{\intt,3}(t,\xi)\big)\big\|_{L^2}\\
	&\qquad\lesssim(1+t)\left\|\chi_{\intt}(\xi)|\xi|^{s+3}\mathrm{e}^{-c|\xi|^2t}\widehat{\Phi}_0(\xi)\right\|_{L^2}+\left\|\chi_{\intt}(\xi)|\xi|^{s+1}\mathrm{e}^{-c|\xi|^2t}\widehat{\Phi}_0(\xi)\right\|_{L^2}\\
	&\qquad\lesssim(1+t)^{-\frac{n+2s}{4}-\frac{1}{2}}\|\Phi_0\|_{L^1}.
\end{align*}
Summarizing the previous estimates, we deduce
\begin{align*}
	\|\Phi(t,\cdot)-U_1\Psi(t,\cdot)\|_{\dot{H}^{s}}&\lesssim \big\||\xi|^{s}\big(S_{\intt,1}(t,\xi)-\chi_{\intt}(\xi)U_1\widehat{\Psi}(t,\xi)\big)\big\|_{L^2}+\big\||\xi|^{s}\big(S_{\intt,2}(t,\xi)+S_{\intt,3}(t,\xi)\big)\big\|_{L^2}\\
	&\quad+\mathrm{e}^{-ct}\|\chi_{\bdd}(\xi)\widehat{\Phi}_0(\xi)\|_{L^2}+\mathrm{e}^{-ct}\|\chi_{\extt}(\xi)|\xi|^{s}\widehat{\Phi}_0(\xi)\|_{L^2}\\
	&\lesssim(1+t)^{-\frac{n+2s}{4}-\frac{1}{2}}\|\Phi_0\|_{L^1}+\mathrm{e}^{-ct}\|\Phi_0\|_{H^{s}},
\end{align*}
which completes the proof of refined estimates.

To derive optimal decay lower bounds, we rewrite the asymptotic profile as
\begin{align*}
|D|^sU_1\Psi(t,x)=|D|^sU_1\diag\big(\mathrm{e}^{\mu_j(|D|)t}\big)_{j=1}^3U_1^{-1}\Phi_0(x)=:\ml{L}_0^{(s)}(t,|D|)\Phi_0(x).
\end{align*}
Note that $\ml{L}_{0,m}^{(s)}(t,x):=\ml{L}_0^{(s)}(t,|D|)$ by applying the Fourier transform. Namely, we already proved
\begin{align*}
\left\||D|^s\Phi(t,\cdot)-\ml{L}_{0,m}^{(s)}(t,\cdot)\Phi_0(\cdot)\right\|_{L^2}\lesssim (1+t)^{-\frac{n+2s}{4}-\frac{1}{2}}\|\Phi_0\|_{H^s\cap L^1}.
\end{align*}
Associated with the fact from the mean value theorem
\begin{align*}
	|\ml{L}_{0,m}^{(s)}(t,x-y)-\ml{L}_{0,m}^{(s)}(t,x)|\lesssim|y|\,|\nabla\ml{L}^{(s)}_{0,m}(t,x-\theta_0y)|\ \ \mbox{with}\ \ \theta_0\in(0,1),
\end{align*}
we may separate the integral into two parts to deduce
\begin{align}\label{Chen-Mod-04}
&\left\|\ml{L}_{0}^{(s)}(t,|D|)\Phi_0(\cdot)-\ml{L}_{0,m}^{(s)}(t,\cdot)P_{\Phi_0}\right\|_{L^2}\notag\\
&\quad\quad\lesssim\left\|\int_{|y|\leqslant t^{\frac{1}{8}}}\big(\ml{L}_{0,m}^{(s)}(t,\cdot-y)-\ml{L}_{0,m}^{(s)}(t,\cdot)\big)\Phi_0(y)\mathrm{d}y\right\|_{L^2}\notag\\
&\quad\quad\quad +\left\|\int_{|y|\geqslant t^{\frac{1}{8}}}\left(|\ml{L}_{0,m}^{(s)}(t,\cdot-y)|+|\ml{L}_{0,m}^{(s)}(t,\cdot)|\right)|\Phi_0(y)|\mathrm{d}y\right\|_{L^2}\notag\\
&\quad\quad\lesssim t^{\frac{1}{8}}\left\|\ml{L}_{0,m}^{(s+1)}(t,\cdot)\right\|_{L^2}\|\Phi_0\|_{L^1}+\left\|\ml{L}_{0,m}^{(s)}(t,\cdot)\right\|_{L^2}\|\Phi_0\|_{L^1(|x|\geqslant t^{\frac{1}{8}})}.
\end{align}
Our additional assumption on the $L^1$ integrability of $\Phi_0$ shows 
\begin{align*}
\|\Phi_0\|_{L^1(|x|\geqslant t^{\frac{1}{8}})}=o(1)\ \ \mbox{for}\ \ t\gg1.	
\end{align*}
According to the well-known optimal decay estimates for the heat kernel and $|\mathrm{e}^{iy}|=1$ for any $y\in\mb{R}$, we claim that
\begin{align*}
\left\|\ml{L}_{0,m}^{(s)}(t,\cdot)\right\|_{L^2}^2\simeq\left\||\xi|^s\mathrm{e}^{\mp i\sqrt{\gamma}c_0|\xi|t}\mathrm{e}^{-\Gamma_0|\xi|^2t}\right\|_{L^2}^2+ \left\||\xi|^s\mathrm{e}^{-D_{\Th}|\xi|^2t}\right\|_{L^2}^2\simeq t^{-\frac{n+2s}{2}}
\end{align*}
for $t\gg1$. Then, the assumption $|P_{\Phi_0}|\neq0$ implies
\begin{align*}
\|\Phi(t,\cdot)\|_{\dot{H}^s}&\gtrsim\left\|\ml{L}_{0,m}^{(s)}(t,\cdot)\right\|_{L^2}|P_{\Phi_0}|-\left\||D|^s\Phi(t,\cdot)-\ml{L}_{0,m}^{(s)}(t,\cdot)\Phi_0(\cdot)\right\|_{L^2}\\
&\quad-\left\|\ml{L}_{0}^{(s)}(t,|D|)\Phi_0(\cdot)-\ml{L}_{0,m}^{(s)}(t,\cdot)P_{\Phi_0}\right\|_{L^2}\\
&\gtrsim t^{-\frac{n+2s}{4}}|P_{\Phi_0}|-t^{-\frac{n+2s}{4}-\frac{1}{2}}\|\Phi_0\|_{H^s\cap L^1}-t^{-\frac{3}{8}-\frac{n+2s}{4}}\|\Phi_0\|_{L^1}-o(t^{-\frac{n+2s}{4}})\\
&\gtrsim t^{-\frac{n+2s}{4}}|P_{\Phi_0}|
\end{align*}
for large time $t\gg1$, which completes our proof of Theorem \ref{Thm-Diag}.

\section{Large time behavior for the velocity potential}\label{Section-Solution-Itself}
$\ \ \ \ $In this section, we will reduce the coupled system \eqref{Thermo-Acoustic-System} to the third order (in time) evolution equation for the velocity potential $\phi$, which did not be included in the energy term $\Phi$. By applying the  WKB analysis and the Fourier analysis, we will derive optimal estimates and optimal leading term of the velocity potential to give the proofs of Theorems \ref{Thm-Optimal-Growth} and \ref{Thm-Optimal-Leading}.
\subsection{Reduction procedure and characteristic roots}
$\ \ \ \ $Acting the operator $\partial_t-\gamma D_{\Th}\Delta$ on the equation \eqref{Thermo-Acoustic-System}$_1$, we may reduce the coupled system \eqref{Thermo-Acoustic-System} in the phase space to the following third order (in time) differential equation with initial conditions:
\begin{align}\label{Thermo-Acoustic-System-third-equ}
\begin{cases}
\widehat{\phi}_{ttt}+[\gamma D_{\Th}+(1+\beta)\nu_0]|\xi|^2\widehat{\phi}_{tt}&\\
\quad\ \ +[\gamma c_0^2|\xi|^2+(1+\beta)\nu_0\gamma D_{\Th}|\xi|^4]\widehat{\phi}_t+
\gamma D_{\Th}c_0^2|\xi|^4\widehat{\phi}=0,&\xi\in\mb{R}^n,\ t>0,\\
\widehat{\phi}(0,\xi)=\widehat{\phi}_0(\xi),\ \widehat{\phi}_t(0,\xi)=\widehat{\phi}_1(\xi),\ \widehat{\phi}_{tt}(0,\xi)=\widehat{\phi}_2(\xi),&\xi\in\mb{R}^n,
\end{cases}
\end{align}
where the third data can be derived by \eqref{Thermo-Acoustic-System-modified} such that
\begin{align}\label{data-third}
    \widehat{\phi}_2(\xi)
    :=-\gamma c_0^2|\xi|^2\widehat{\phi}_0(\xi)-(1+\beta)\nu_0|\xi|^2\widehat{\phi}_1(\xi)+\alpha_p c_0^2\gamma D_{\Th}|\xi|^2\widehat{T}_0(\xi).
\end{align}
The corresponding characteristic equation to \eqref{Thermo-Acoustic-System-third-equ} is given by the $|\xi|$-dependent cubic equation
\begin{align}\label{characterictic-equ}
    \lambda^3+\left[\gamma D_{\Th}+(1+\beta)\nu_0\right]|\xi|^2\lambda^2+[\gamma c_0^2|\xi|^2+(1+\beta)\nu_0\gamma D_{\Th}|\xi|^4]\lambda+
\gamma D_{\Th}c_0^2|\xi|^4=0.  
\end{align}
 Taylor-like expansions with respect to $|\xi|$ will be used as $\xi\in\ml{Z}_{\intt}(\varepsilon_0)$ with $\varepsilon_0\ll1 $. Then, the roots $\lambda_j=\lambda_j(|\xi|)$ with $j=1,2,3$ to the $|\xi|$-dependent cubic \eqref{characterictic-equ} can be expanded straightforwardly
 	 \begin{align*}
 		\lambda_1(|\xi|)&=-D_{\Th}|\xi|^2+\frac{D_{\Th}^2(\gamma-1)\left[(1+\beta)\nu_0-D_{\Th}\right]}{\gamma c_0^2}|\xi|^4+\ml{O}(|\xi|^5),\\
 		\lambda_{2,3}(|\xi|)&=\mp i\sqrt{\gamma}c_0|\xi|-\Gamma_0|\xi|^2\mp i\Gamma_1|\xi|^3+\ml{O}(|\xi|^4),
 	\end{align*}
in which two constants $\Gamma_0$ and $\Gamma_1$ are defined in \eqref{G0G1}.
\begin{remark}
For $\xi\in\ml{Z}_{\intt}(\varepsilon_0)$, comparing the last expansions with those in Proposition \ref{Prop-small-frenquencies}, we explicitly obtained the higher order $|\xi|^4$-term in $\lambda_1(|\xi|)$ and $|\xi|^3$-terms in $\lambda_{2,3}(|\xi|)$. These higher order factors will contribute to the study of optimal leading terms.
\end{remark}

\subsection{Refined estimates in the phase space}
$\ \ \ \ $In this subsection, we will give asymptotic expressions of the solution in the phase space and build some auxiliary functions to analyze asymptotic behavior of the leading term. For the sake of briefness, we will omit the variable $\xi$ sometimes of the unknown functions and their initial value. 

Due to the fact that the discriminant of the characteristic equation \eqref{characterictic-equ} is strictly negative for $\xi\in\ml{Z}_{\intt}(\varepsilon_0)$, the characteristic roots $\lambda_{2,3}$ are complex conjugate. We set $\lambda_{2,3}=\lambda_{\mathrm{R}} \pm i\lambda_{\mathrm{I}}$ for small frequencies carrying
\begin{align*}
    \lambda_{\mathrm{R}}=-\Gamma_0|\xi|^2+\ml{O}(|\xi|^4),\quad \lambda_{\mathrm{I}}=-\sqrt{\gamma}c_0|\xi|-\Gamma_1|\xi|^3+\ml{O}(|\xi|^4).
\end{align*}
Then, via basic computations by the ODE theory with constant coefficients and the pairwise distinct characteristic roots, the solution for $\xi\in\ml{Z}_{\intt}(\varepsilon_0)$ can be expressed by
\begin{align*}
     \widehat{\phi}&=\frac{(\gamma c_0^2|\xi|^2-\lambda_{\mathrm{I}}^2-\lambda_{\mathrm{R}}^2)\widehat{\phi}_0+[2\lambda_{\mathrm{R}}+(1+\beta)\nu_0|\xi|^2]\widehat{\phi}_1-\alpha_p c_0^2\gamma D_{\Th}|\xi|^2\widehat{T}_0}{2\lambda_{\mathrm{R}}\lambda_1-\lambda_{\mathrm{I}}^2-\lambda_{\mathrm{R}}^2-\lambda_1^2}\mathrm{e}^{\lambda_1t}\\
     &\quad+\frac{(2\lambda_{\mathrm{R}}\lambda_1-\lambda_1^2-\gamma c_0^2|\xi|^2)\widehat{\phi}_0-[2\lambda_{\mathrm{R}}+(1+\beta)\nu_0|\xi|^2]\widehat{\phi}_1+\alpha_p c_0^2\gamma D_{\Th}|\xi|^2\widehat{T}_0}{2\lambda_{\mathrm{R}}\lambda_1-\lambda_{\mathrm{I}}^2-\lambda_{\mathrm{R}}^2-\lambda_1^2}\cos(\lambda_{\mathrm{I}}t)\mathrm{e}^{\lambda_{\mathrm{R}}t}\\
     &\quad+\frac{[\lambda_1(\lambda_{\mathrm{R}}\lambda_1+\lambda_{\mathrm{I}}^2-\lambda_{\mathrm{R}}^2)+\gamma c_0^2|\xi|^2(\lambda_{\mathrm{R}}-\lambda_1)]\widehat{\phi}_0}{\lambda_{\mathrm{I}}(2\lambda_{\mathrm{R}}\lambda_1-\lambda_{\mathrm{I}}^2-\lambda_{\mathrm{R}}^2-\lambda_1^2)}\sin(\lambda_{\mathrm{I}}t)\mathrm{e}^{\lambda_{\mathrm{R}}t}\\
     &\quad+\frac{[\lambda_{\mathrm{R}}^2-\lambda_{\mathrm{I}}^2-\lambda_1^2+(\lambda_{\mathrm{R}}-\lambda_1)(1+\beta)\nu_0|\xi|^2]\widehat{\phi}_1-(\lambda_{\mathrm{R}}-\lambda_1)\alpha_p c_0^2\gamma D_{\Th}|\xi|^2\widehat{T}_0}{\lambda_{\mathrm{I}}(2\lambda_{\mathrm{R}}\lambda_1-\lambda_{\mathrm{I}}^2-\lambda_{\mathrm{R}}^2-\lambda_1^2)}\sin(\lambda_{\mathrm{I}}t)\mathrm{e}^{\lambda_{\mathrm{R}}t},
\end{align*}
where the third data \eqref{data-third} was used.
To analyze its asymptotic behavior, due to the power of $|\xi|$, we introduce the auxiliary functions $\widehat{J}_0=\widehat{J}_0(t,\xi)$ and $\widehat{J}_1=\widehat{J}_1(t,\xi)$ as follows:
\begin{align*}
    \widehat{J}_0&:=-\frac{\lambda_{\mathrm{I}}\widehat{\phi}_1}{2\lambda_{\mathrm{R}}\lambda_1-\lambda_{\mathrm{I}}^2-\lambda_{\mathrm{R}}^2-\lambda_1^2}\sin(\lambda_{\mathrm{I}}t)\mathrm{e}^{\lambda_{\mathrm{R}}t},\\
    \widehat{J}_1	&:=\frac{(\gamma c_0^2|\xi|^2-\lambda_{\mathrm{I}}^2)\mathrm{e}^{\lambda_1t}-\gamma c_0^2|\xi|^2\cos(\lambda_{\mathrm{I}}t)\mathrm{e}^{\lambda_{\mathrm{R}}t}}{2\lambda_{\mathrm{R}}\lambda_1-\lambda_{\mathrm{I}}^2-\lambda_{\mathrm{R}}^2-\lambda_1^2}\widehat{\phi}_0+\frac{[2\lambda_{\mathrm{R}}+(1+\beta)\nu_0|\xi|^2]\widehat{\phi}_1}{2\lambda_{\mathrm{R}}\lambda_1-\lambda_{\mathrm{I}}^2-\lambda_{\mathrm{R}}^2-\lambda_1^2}\left(\mathrm{e}^{\lambda_1t}-\cos(\lambda_{\mathrm{I}}t)\mathrm{e}^{\lambda_{\mathrm{R}}t}\right)\\
    &\quad\ \ -\frac{\alpha_p c_0^2\gamma D_{\Th}|\xi|^2\widehat{T}_0}{2\lambda_{\mathrm{R}}\lambda_1-\lambda_{\mathrm{I}}^2-\lambda_{\mathrm{R}}^2-\lambda_1^2}\left(\mathrm{e}^{\lambda_1t}-\cos(\lambda_{\mathrm{I}}t)\mathrm{e}^{\lambda_{\mathrm{R}}t}\right).
\end{align*}

Next, we propose some estimates for the above auxiliary functions and estimates of the solution by subtracting these auxiliary functions.
\begin{prop}\label{prop-auxilary-functions-1}
    The auxiliary functions and the error terms satisfy the following estimates:
    \begin{align}
       \chi_{\intt}(\xi)|\widehat{J}_0|&\lesssim \chi_{\intt}(\xi)\frac{|\sin(\sqrt{\gamma}c_0|\xi|t)|}{|\xi|}\mathrm{e}^{-c|\xi|^2t}|\widehat{\phi}_1|,\label{Chen-Mod-01}\\
       \chi_{\intt}(\xi)\left(|\widehat{J}_1|+|\widehat{\phi}-\widehat{J}_0|\right)&\lesssim\chi_{\intt}(\xi)\mathrm{e}^{-c|\xi|^2t}\left(|\widehat{\phi}_0|+|\widehat{\phi}_1|+|\widehat{T}_0|\right),\notag\\
       \chi_{\intt}(\xi)|\widehat{\phi}-\widehat{J}_0-\widehat{J}_1|&\lesssim\chi_{\intt}(\xi)|\xi|\mathrm{e}^{-c|\xi|^2t}\left(|\widehat{\phi}_0|+|\widehat{\phi}_1|+|\widehat{T}_0|\right),\label{Chen-Mod-02}
    \end{align}
as well as the solution fulfills the following estimate:
\begin{align*}
       \chi_{\intt}(\xi)|\widehat{\phi}|\lesssim\chi_{\intt}(\xi)\mathrm{e}^{-c|\xi|^2t}\left[|\widehat{\phi}_0|+
\left(1+\frac{|\sin(\sqrt{\gamma}c_0|\xi|t)|}{|\xi|}\right)|\widehat{\phi}_1|+|\widehat{T}_0|\right].
\end{align*}
\end{prop}
\begin{proof}
    By substituting the expansions of $\lambda_1,\lambda_{\mathrm{R}}$ and $\lambda_{\mathrm{I}}$ into the auxiliary function $\widehat{J}_0$ and calculating the leading term for $\xi\in\ml{Z}_{\intt}(\varepsilon_0)$, we  may derive the estimate \eqref{Chen-Mod-01} easily. To estimate  the auxiliary function $\widehat{J}_1$, we need to use the following equality:
    \begin{align*}
        \mathrm{e}^{\lambda_1t}-\cos(\lambda_{\mathrm{I}}t)\mathrm{e}^{\lambda_{\mathrm{R}}t}&=\left(\mathrm{e}^{\lambda_1t}-\mathrm{e}^{\lambda_{\mathrm{R}}t}\right)+[1-\cos(\lambda_{\mathrm{I}}t)]\mathrm{e}^{\lambda_{\mathrm{R}}t}\\
        &=(\lambda_1-\lambda_{\mathrm{R}})t\,\mathrm{e}^{\lambda_{\mathrm{R}}t}\int_{0}^{1}\mathrm{e}^{(\lambda_1-\lambda_{\mathrm{R}})t\tau}\mathrm{d}\tau +2\sin^2\left(\frac{1}{2}\lambda_{\mathrm{I}}t\right)\mathrm{e}^{\lambda_{\mathrm{R}}t}.
    \end{align*}
    Consequently, one arrives at
    \begin{align*}
        \chi_{\intt}(\xi)|\widehat{J}_1|&\lesssim \chi_{\intt}(\xi)\big[1+|\cos(\sqrt{\gamma}c_0|\xi|t)|\big]\mathrm{e}^{-c|\xi|^2t}|\widehat{\phi}_0|\\
        &\quad+\chi_{\intt}(\xi)\left|(\lambda_1-\lambda_{\mathrm{R}})t\,\mathrm{e}^{\lambda_{\mathrm{R}}t}\int_{0}^{1}\mathrm{e}^{(\lambda_1-\lambda_{\mathrm{R}})t\tau}\mathrm{d}\tau +2\sin^2\left(\frac{1}{2}\lambda_{\mathrm{I}}t\right)\mathrm{e}^{\lambda_{\mathrm{R}}t}\right|\left(|\widehat{\phi}_1|+|\widehat{T}_0|\right) \\
        &\lesssim \chi_{\intt}(\xi)\mathrm{e}^{-c|\xi|^2t}|\widehat{\phi}_0|+\chi_{\intt}(\xi)\left(|\xi|^2t\, \mathrm{e}^{-c|\xi|^2t}+\left|\sin^2\left(\frac{1}{2}\lambda_{\mathrm{I}}t\right)\right|\mathrm{e}^{-c|\xi|^2t}\right)\left(|\widehat{\phi}_1|+|\widehat{T}_0|\right)\\
        &\lesssim \chi_{\intt}(\xi)\mathrm{e}^{-c|\xi|^2t}\left(|\widehat{\phi}_0|+|\widehat{\phi}_1|+|\widehat{T}_0|\right),
    \end{align*}
    where we used the boundedness of trigonometric functions.

     Since $\widehat{J}_0$ and $\widehat{J}_1$ are chosen as the dominant terms of $\chi_{\intt}(\xi)\widehat{\phi}(t,\xi)$ for the first and second orders, respectively, according to the representation, we can get a refined estimate \eqref{Chen-Mod-02} with an additional factor $|\xi|^2$ by subtracting the leading terms.

    Finally, from the triangle inequality we conclude
    \begin{align*}
          \chi_{\intt}(\xi)|\widehat{\phi}-\widehat{J}_0|&\lesssim\chi_{\intt}(\xi)|\widehat{\phi}-\widehat{J}_0-\widehat{J}_1|+\chi_{\intt}(\xi)|\widehat{J}_1|\\
          &\lesssim\chi_{\intt}(\xi)\mathrm{e}^{-c|\xi|^2t}\left(|\widehat{\phi}_0|+|\widehat{\phi}_1|+|\widehat{T}_0|\right),
    \end{align*}
and
    \begin{align*}
          \chi_{\intt}(\xi)|\widehat{\phi}|&\lesssim\chi_{\intt}(\xi)|\widehat{\phi}-\widehat{J}_0|+\chi_{\intt}(\xi)|\widehat{J}_0|\\
          &\lesssim\chi_{\intt}(\xi)\mathrm{e}^{-c|\xi|^2t}\left[|\widehat{\phi}_0|+\left(1+\frac{|\sin(\sqrt{\gamma}c_0|\xi|t)|}{|\xi|}\right)|\widehat{\phi}_1|+|\widehat{T}_0|\right].
    \end{align*}
    Then, we complete the proof of this proposition.
\end{proof}

 We extract the dominant terms in the auxiliary functions $\widehat{J}_0$ and $\widehat{J}_1$. Precisely,  we take the Fourier multipliers $\widehat{\ml{G}}_j=\widehat{\ml{G}}_j(t,|\xi|)$ for $j=0,\dots,3$ and $\widehat{\ml{H}}_0=\widehat{\ml{H}}_0(t,|\xi|)$ via
\begin{align*}
    \widehat{\ml{G}}_0:&=\frac{\sin(\sqrt{\gamma}c_0|\xi|t)}{\sqrt{\gamma}c_0|\xi|}\mathrm{e}^{-\Gamma_0|\xi|^2t},\quad \,\,
        \widehat{\ml{H}}_0:=\frac{\Gamma_1}{\sqrt{\gamma}c_0}\cos(\sqrt{\gamma}c_0|\xi|t)|\xi|^2t\,\mathrm{e}^{-\Gamma_0|\xi|^2t},\\
   \widehat{\ml{G}}_1:&=\cos\left(\sqrt{\gamma}c_0|\xi|t\right)\mathrm{e}^{-\Gamma_0|\xi|^2t},\quad \widehat{\ml{G}}_2:=\frac{(\gamma-1)D_{\Th}}{\gamma c_0^2}\left(\mathrm{e}^{-D_{\Th}|\xi|^2t}-\cos\left(\sqrt{\gamma}c_0|\xi|t\right)\mathrm{e}^{-\Gamma_0|\xi|^2t}\right),\\
    \widehat{\ml{G}}_3:&=\alpha_p D_{\Th}\left(\mathrm{e}^{-D_{\Th}|\xi|^2t}-\cos\left(\sqrt{\gamma}c_0|\xi|t\right)\mathrm{e}^{-\Gamma_0|\xi|^2t}\right),
\end{align*}
which are the Fourier transforms of the auxiliary functions introduced in Section \ref{Section_Main_Result}. Then, some refined estimates contributing to estimate the $L^2$ norm of $\phi(t,\cdot)$  can be established.
\begin{prop}\label{prop-auxilary-functions-2}
    The error terms satisfy the following estimates:
    \begin{align*}
       \chi_{\intt}(\xi)|\widehat{J}_0-\widehat{\ml{G}}_0\widehat{\phi}_1|&\lesssim \chi_{\intt}(\xi)\mathrm{e}^{-c|\xi|^2t}|\widehat{\phi}_1|,\\
       \chi_{\intt}(\xi)|\widehat{J}_0-(\widehat{\ml{G}}_0+\widehat{\ml{H}}_0)\widehat{\phi}_1|&\lesssim\chi_{\intt}(\xi)|\xi|\mathrm{e}^{-c|\xi|^2t}|\widehat{\phi}_1|.
    \end{align*}
\end{prop}
\begin{proof}
    From the representations of $\widehat{J}_0$ and $\widehat{\ml{G}}_0$, the error term can be decomposed into three parts
    \begin{align*}
      \widehat{J}_0-\widehat{\ml{G}}_0\widehat{\phi}_1&=\left(\frac{-\lambda_{\mathrm{I}}}{2\lambda_{\mathrm{R}}\lambda_1-\lambda_{\mathrm{I}}^2-\lambda_{\mathrm{R}}^2-\lambda_1^2}+\frac{1}{\sqrt{\gamma}c_0|\xi|}\right)\sin(\lambda_{\mathrm{I}}t)\mathrm{e}^{\lambda_{\mathrm{R}}t}\widehat{\phi}_1+\frac{-\sin(\lambda_{\mathrm{I}}t)-\sin(\sqrt{\gamma}c_0|\xi|t)}{\sqrt{\gamma}c_0|\xi|}\mathrm{e}^{\lambda_{\mathrm{R}}t}\widehat{\phi}_1\\
      &\quad\ +\frac{\sin(\sqrt{\gamma}c_0|\xi|t)}{\sqrt{\gamma}c_0|\xi|}\left(\mathrm{e}^{\lambda_{\mathrm{R}}t}-\mathrm{e}^{-\Gamma_0|\xi|^2t}\right)\widehat{\phi}_1\\
      &=:\widehat{I}_{0,1}\widehat{\phi}_1+\widehat{I}_{0,2}\widehat{\phi}_1+\widehat{I}_{0,3}\widehat{\phi}_1.
     \end{align*}
     Direct computations find $\lambda_{\mathrm{I}}+\sqrt{\gamma}c_0|\xi|=-\Gamma_1|\xi|^3+\ml{O}(|\xi|^4)$ benefited from higher order expansions of the characteristic roots. We arrive at
     \begin{align*}
         \chi_{\intt}(\xi)|\widehat{I}_{0,1}|&\lesssim\chi_{\intt}(\xi)\left|\frac{-\lambda_{\mathrm{I}}(\lambda_{\mathrm{I}}+\sqrt{\gamma}c_0|\xi|)+2\lambda_{\mathrm{R}}\lambda_1-\lambda_{\mathrm{R}}^2-\lambda_1^2}{(2\lambda_{\mathrm{R}}\lambda_1-\lambda_{\mathrm{I}}^2-\lambda_{\mathrm{R}}^2-\lambda_1^2)\sqrt{\gamma}c_0|\xi|}\right|\mathrm{e}^{\lambda_{\mathrm{R}}t}\lesssim\chi_{\intt}(\xi)|\xi|\mathrm{e}^{-c|\xi|^2t}.
     \end{align*}
     Concerning $\widehat{I}_{0,2}$, we may employ the mean value theorem as $\xi\in\ml{Z}_{\intt}(\varepsilon_0)$ to obtain that there is $\widetilde{m}=\widetilde{m}(|\xi|)$ belonging to $(0,\Gamma_1|\xi|^3+\ml{O}(|\xi|^4))$ such that
     \begin{align}\label{Chen-Mod-03}
         &\sin(-\lambda_{\mathrm{I}}t)-\sin(\sqrt{\gamma}c_0|\xi|t)\notag\\
         &\qquad\quad= \left[\Gamma_1|\xi|^3t+\ml{O}(|\xi|^4)t\right]\cos(\sqrt{\gamma}c_0|\xi|t)-\frac{\sin(\sqrt{\gamma}c_0|\xi|t\,\widetilde{m})}{2}\left[\Gamma_1|\xi|^3t+\ml{O}(|\xi|^4)t\right]^2,
     \end{align}
     which means 
     \begin{align*}
         \chi_{\intt}(\xi)|\widehat{I}_{0,2}|\lesssim\chi_{\intt}(\xi)|\xi|^2t\,\mathrm{e}^{-c|\xi|^2t}\lesssim\chi_{\intt}(\xi)\mathrm{e}^{-c|\xi|^2t}.
     \end{align*}
     Analogously,  we can control $\widehat{I}_{0,3}$ by using the formulate $\lambda_{\mathrm{R}}+\Gamma_0|\xi|^2=\ml{O}(|\xi|^4)$ as follows:
     \begin{align*}
         \chi_{\intt}(\xi)|\widehat{I}_{0,3}|
         \lesssim\chi_{\intt}(\xi)|\xi|^3t\,\mathrm{e}^{-\Gamma_0|\xi|^2t}\int_0^1\mathrm{e}^{\ml{O}(|\xi|^4)t\tau}\mathrm{d}\tau\lesssim\chi_{\intt}(\xi)|\xi|\mathrm{e}^{-c|\xi|^2t}.
     \end{align*}
     Above all, we summarize the obtained estimates to complete the first part of the proposition. 
     
    With the aim of deducing a further refined estimate, we employ the representations of auxiliary functions to get
     \begin{align*}
         \widehat{J}_0-(\widehat{\ml{G}}_0+\widehat{\ml{H}}_0)\widehat{\phi}_1&=\widehat{I}_{0,1}\widehat{\phi}_1+\underbrace{\left(\widehat{I}_{0,2}-\frac{\Gamma_1}{\sqrt{\gamma}c_0}|\xi|^2t\cos\left(\sqrt{\gamma}c_0|\xi|t\right)\mathrm{e}^{-\Gamma_0|\xi|^2t}\right)}_{=:\widehat{I}_{0,2}'}\widehat{\phi}_1+\widehat{I}_{0,3}\widehat{\phi}_1.
     \end{align*}
     We remark that $\widehat{I}_{0,1}$ and $\widehat{I}_{0,3}$ were estimated in the previous calculation, so we only need to control the error term. Thanks to the expressions of $\widehat{\ml{G}}_0$ and $\widehat{\ml{H}}_0$, we obtain
     \begin{align*}
         \chi_{\intt}(\xi)|\widehat{I}_{0,2}'|&\lesssim\chi_{\intt}(\xi) \left|\frac{-\sin(\lambda_{\mathrm{I}}t)-\sin(\sqrt{\gamma}c_0|\xi|t)-\Gamma_1|\xi|^3t\cos(\sqrt{\gamma}c_0|\xi|t)}{\sqrt{\gamma} c_0|\xi|}\right|\mathrm{e}^{\lambda_{\mathrm{R}}t}\\
         &\quad+\chi_{\intt}(\xi)\left|\frac{\Gamma_1}{\sqrt{\gamma}c_0}|\xi|^2t\cos(\sqrt{\gamma}c_0|\xi|t)\right|\left|\mathrm{e}^{\lambda_{\mathrm{R}}t}-\mathrm{e}^{-\Gamma_0|\xi|^2t}\right|\\
         &\lesssim \chi_{\intt}(\xi)|\xi|^5t^2\mathrm{e}^{-c|\xi|^2t}\lesssim \chi_{\intt}(\xi)|\xi|\mathrm{e}^{-c|\xi|^2t},
     \end{align*}
     where we used the formulate \eqref{Chen-Mod-03} again. Summarizing the above estimates, we complete the proof.
\end{proof}

By following the same approaches as those for Proposition \ref{prop-auxilary-functions-2}, due to the subtraction with the leading terms of $\widehat{J}_1$, we may derive a further refined estimate with an additional factor $|\xi|$ comparing with the estimate of $\widehat{J}_1$ in Proposition \ref{prop-auxilary-functions-1}, in which the proof considers the mean value theorem of the cosine function instead of the sine function. 
\begin{prop}\label{prop-auxilary-functions-3}
    The error term  satisfies the following estimate:
    \begin{align*}
       \chi_{\intt}(\xi)|\widehat{J}_1-\widehat{\ml{G}}_1\widehat{\phi}_0-\widehat{\ml{G}}_2\widehat{\phi}_1-\widehat{\ml{G}}_3\widehat{T}_0|&\lesssim\chi_{\intt}(\xi)|\xi|\mathrm{e}^{-c|\xi|^2t}\left(|\widehat{\phi}_0|+|\widehat{\phi}_1|+|\widehat{T}_0|\right).
    \end{align*}
\end{prop}

To end this preparation, let us analyze pointwise estimates in the phase space for $\xi\in\ml{Z}_{\bdd}(\varepsilon_0,N_0)\cup\ml{Z}_{\extt}(N_0)$. According to \eqref{pointwise-est-Phi} and the expressions of $\widehat{\Phi}(t,\xi)$ as well as $\widehat{\Phi}_0(\xi)$, we immediately obtain
\begin{align}\label{Est-ext}
\chi_{\extt}(\xi)|\widehat{\phi}|&\lesssim\chi_{\extt}(\xi)\mathrm{e}^{-ct}\left(|\widehat{\phi}_0|+|\xi|^{-1}|\widehat{\phi}_1|+|\xi|^{-1}|\widehat{T}_0|\right),\\
    \chi_{\bdd}(\xi)|\widehat{\phi}|&\lesssim\chi_{\bdd}(\xi)\mathrm{e}^{-ct}\left(|\widehat{\phi}_0|+|\widehat{\phi}_1|+|\widehat{T}_0|\right).\label{Est-bdd}
 \end{align}

\subsection{Optimal growth/decay estimates: Proof of Theorem \ref{Thm-Optimal-Growth}}
$\ \ \ \ $First of all, for small frequencies, Proposition \ref{prop-auxilary-functions-1} shows
\begin{align*}
    \|\chi_{\intt}(\xi)\widehat{\phi}(t,\xi)\|_{L^2}\lesssim\left\|\chi_{\intt}(\xi)\mathrm{e}^{-c|\xi|^2t}\left(|\widehat{\phi}_0(\xi)|+|\widehat{T}_0(\xi)|\right)\right\|_{L^2}+\left\|\chi_{\intt}(\xi)\frac{|\sin(\sqrt{\gamma}c_0|\xi|t)|}{|\xi|}\mathrm{e}^{-c|\xi|^2t}|\widehat{\phi}_1(\xi)|\right\|_{L^2}.
\end{align*}
Recalling some preliminaries in \cite{Ikehata=2014,Ikehata-Onodera=2017}, the optimal estimates   for $t\gg1$ are as follows:
\begin{align}\label{Chen-Mod-05}
    \left\|\mathrm{e}^{-c|\xi|^2t}\right\|_{L^2}\simeq t^{-\frac{n}{4}}\ \ \mbox{as well as}\ \ \left\|\frac{|\sin(|\xi|t)|}{|\xi|}\mathrm{e}^{-c|\xi|^2t}\right\|_{L^2}\simeq \ml{A}_n(t),
\end{align}
in which the time-dependent function $\ml{A}_n(t)$ was defined in \eqref{A(t)}. Then, by using H\"older's inequality and the Hausdorff-Young inequality one deduces
\begin{align*}
    \|\chi_{\intt}(\xi)\widehat{\phi}(t,\xi)\|_{L^2}\lesssim t^{-\frac{n}{4}}\|(\phi_0,\phi_1,T_0)\|_{(L^1)^3}+\ml{A}_n(t)\|\phi_1\|_{L^1}.
\end{align*}
For bounded frequencies and large frequencies, \eqref{Est-ext} and \eqref{Est-bdd} imply
\begin{align*}
    \left\|\big(\chi_{\bdd}(\xi)+\chi_{\extt}(\xi)\big)\widehat{\phi}(t,\xi)\right\|_{L^2}\lesssim \mathrm{e}^{-ct}\|(\phi_0,\phi_1,T_0)\|_{(L^2)^3}.
\end{align*}
Employing the Plancherel theorem associated with the previous estimates, we may obtain the desired upper bound estimates in \eqref{Est-phi-Est}.

Concerning the sharp lower bound estimates for large time, we first use the triangle inequality combined with Propositions \ref{prop-auxilary-functions-1} and \ref{prop-auxilary-functions-2} to find
\begin{align}\label{upper-int-phi-Gphi1}
    &\big\|\chi_{\intt}(D)\big(\phi(t,\cdot)-\ml{G}_0(t,|D|)\phi_1(\cdot)\big)\big\|_{L^2}\notag\\
    &\qquad\lesssim \big\|\chi_{\intt}(\xi)\big(\widehat{\phi}(t,\xi)-\widehat{J}_0(t,\xi)\big)\big\|_{L^2}+\big\|\chi_{\intt}(\xi)\big(\widehat{J}_0(t,\xi)-\widehat{\ml{G}}_0(t,|\xi|)\widehat{\phi}_1(\xi)\big)\big\|_{L^2}\notag\\
    &\qquad\lesssim \left\|\chi_{\intt}(\xi)\,\mathrm{e}^{-c|\xi|^2t}\left(|\widehat{\phi}_0(\xi)|+|\widehat{\phi}_1(\xi)|+|\widehat{T}_0(\xi)|\right)\right\|_{L^2}\notag\\
    &\qquad\lesssim t^{-\frac{n}{4}}\|(\phi_0,\phi_1,T_0)\|_{(L^1)^3}.
\end{align}
From \eqref{Est-ext}, \eqref{Est-bdd} and the representation of $\widehat{\ml{G}}_0$, we observe that 
\begin{align*}
    \left\|\big(1-\chi_{\intt}(D)\big)\big(\phi(t,\cdot)-\ml{G}_0(t,|D|)\phi_1(\cdot)\big)\right\|_{L^2}\lesssim \mathrm{e}^{-ct}\|(\phi_0,\phi_1,T_0)\|_{(L^2)^3}.
\end{align*}
Therefore, by a suitable decomposition
\begin{align*}
    \phi(t,x)-\ml{G}_0(t,x)P_{\phi_1}=\big(\phi(t,x)-\ml{G}_0(t,|D|)\phi_1(x)\big)+\big(\ml{G}_0(t,|D|)\phi_1(x)-\ml{G}_0(t,x)P_{\phi_1}\big),
\end{align*}
we immediately claim the next estimate for large time:
\begin{align*}
    \|\phi(t,\cdot)-\ml{G}_0(t,\cdot)P_{\phi_1}\|_{L^2}\lesssim t^{-\frac{n}{4}}\|(\phi_0,\phi_1,T_0)\|_{(L^2\cap L^1)^3}+\|\ml{G}_0(t,|D|)\phi_1(\cdot)-\ml{G}_0(t,\cdot)P_{\phi_1}\|_{L^2}.
\end{align*}
Similarly to the approach of \eqref{Chen-Mod-04}, we are able to derive large time error estimates
\begin{align}\label{Chen-Mod-06}
    \|\ml{G}_0(t,|D|)\phi_1(\cdot)-\ml{G}_0(t,\cdot)P_{\phi_1}\|_{L^2}&\lesssim t^{\frac{1}{8}}\|\ml{G}_0(t,\cdot)\|_{\dot{H}^1}\|\phi_1\|_{L^1}+\|\ml{G}_0(t,\cdot)\|_{L^2}\|\phi_1\|_{L^1(|x|\geqslant t^{\frac{1}{8}})}\notag\\
    &\lesssim t^{\frac{1}{8}-\frac{n}{4}}\|\phi_1\|_{L^1}+o\big(\ml{A}_n(t)\big),
\end{align}
thanks to our hypothesis $\phi_1\in L^1$. From the above derived estimates and Minkowski's inequality for $t\gg1$ we deduce
\begin{align*}
    \|\phi(t,\cdot)\|_{L^2}&\gtrsim \|\ml{G}_0(t,\cdot)P_{\phi_1}\|_{L^2}-\|\phi(t,\cdot)-\ml{G}_0(t,\cdot)P_{\phi_1}\|_{L^2}\\&\gtrsim \ml{A}_n(t)|P_{\phi_1}|-t^{-\frac{n}{4}}\|(\phi_0,\phi_1,T_0)\|_{(L^2\cap L^1)^3}-t^{\frac{1}{8}-\frac{n}{4}}\|\phi_1\|_{L^1}-o\big(\ml{A}_n(t)\big)\\&\gtrsim \ml{A}_n(t)|P_{\phi_1}|,
\end{align*}
due to the assumption $|P_{\phi_1}|\neq 0$ and $\|\ml{G}_0(t,\cdot)\|_{L^2}\gtrsim\ml{A}_n(t)$ from \eqref{Chen-Mod-05}.
We now complete the proof by combining the derived estimates of the upper and lower bounds.

\subsection{Optimal leading term and second large time profile: Proof of Theorem \ref{Thm-Optimal-Leading}}
\textbf{\underline{Step 1. Upper bound estimates for the error term}}: Let us apply  $|\widehat{f}-P_f|\lesssim|\xi|\,\|f\|_{L^{1,1}}$ from \cite{Ikehata=2014} to improve the decay rate of \eqref{Chen-Mod-06} by assuming $\phi_1\in L^{1,1}$ additionally. Then, combining the resultant with \eqref{upper-int-phi-Gphi1}, we may derive
\begin{align*}
    \big\|\chi_{\intt}(D)\big(\phi(t,\cdot)-\ml{G}_0(t,\cdot)P_{\phi_1}\big)\big\|_{L^2}&\lesssim t^{-\frac{n}{4}}\|(\phi_0,\phi_1,T_0)\|_{(L^1)^3}+\left\|\chi_{\intt}(\xi)\widehat{\ml{G}}_0(t,|\xi|)\big(\widehat{\phi}_1(\xi)-P_{\phi_1}\big)\right\|_{L^2}\\
    &\lesssim t^{-\frac{n}{4}}\|(\phi_0,\phi_1,T_0)\|_{(L^1)^3}+\left\|\chi_{\intt}(\xi)|\xi|\widehat{\ml{G}}_0(t,|\xi|)\right\|_{L^2} \|\phi_1\|_{L^{1,1}}\\
    &\lesssim t^{-\frac{n}{4}}\|(\phi_0,\phi_1,T_0)\|_{L^1 \times L^{1,1}\times L^1},
\end{align*}
where we used \eqref{Chen-Mod-05} again. From the exponential decay estimate for bounded frequencies and large frequencies 
\begin{align*}
    \big\|\big(\chi_{\bdd}(D)+\chi_{\extt}(D)\big)\big(\phi(t,\cdot)-\ml{G}_0(t,\cdot)P_{\phi_1}\big)\big\|_{L^2}&\lesssim \mathrm{e}^{-ct}\|(\phi_0,\phi_1,T_0)\|_{(L^2)^3}+\mathrm{e}^{-ct}|P_{\phi_1}|
\end{align*}
and the fact that $|P_{\phi_1}|\lesssim \|\phi_1\|_{L^{1,1}}$ we conclude
\begin{align*}
    \|\phi(t,\cdot)-\ml{G}_0(t,\cdot)P_{\phi_1}\|_{L^2}\lesssim t^{-\frac{n}{4}}\left\|\left(\phi_0,\phi_1,T_0\right)\right\|_{(L^2\cap L^1)\times(L^2\cap L^{1,1})\times (L^2\cap L^1)}
\end{align*}
for large time $t\gg1$. Recall that $\varphi(t,x)=\ml{G}_0(t,x)P_{\phi_1}$ from the definition \eqref{first-leading-term} to be the leading term of the acoustic velocity potential $\phi(t,x)$. The upper bound estimate in \eqref{Est-phi-varphi-Est} is completed.
\medskip

\noindent\textbf{\underline{Step 2. Upper bound estimates for the further error term}}: In order to show the sharpness of the leading term, we will estimate the error $\phi(t,\cdot)-\ml{G}_0(t,\cdot)P_{\phi_1}$ in the $L^2$ norm from the sharp lower bound perspective. Thus, the first step is to investigate a second asymptotic profile for large time. Recalling the function $\psi=\psi(t,x)$ defined in \eqref{psi-function}, we decompose the further error term into five parts as follows:
\begin{align*}
    \phi(t,x)-\varphi(t,x)-\psi(t,x)&=\sum\limits_{k=1,\dots,5}E_k(t,x)
\end{align*}
whose components are
\begin{align*}
    E_1(t,x)&:=\phi(t,x)-\ml{G}_0(t,|D|)\phi_1(x)-\ml{G}_1(t,|D|)\phi_0(x)\\
    &\quad\ -\big(\ml{H}_0(t,|D|)+\ml{G}_2(t,|D|)\big)\phi_1(x) -\ml{G}_3(t,|D|)T_0(x),
\end{align*}
and
\begin{align*}
	E_2(t,x)&:=\ml{G}_0(t,|D|)\phi_1(x)-\ml{G}_0(t,x)P_{\phi_1}+\nabla\ml{G}_0(t,x)\circ M_{\phi_1},\\
E_3(t,x)&:=\ml{G}_1(t,|D|)\phi_0(x)-\ml{G}_1(t,x)P_{\phi_0},\\
E_4(t,x)&:=\big(\ml{H}_0(t,|D|)+\ml{G}_2(t,|D|)\big)\phi_1(x)-\big(\ml{H}_0(t,x)+\ml{G}_2(t,x)\big)P_{\phi_1},\\
E_5(t,x)&:=\ml{G}_3(t,|D|)T_0(x)-\ml{G}_3(t,x)P_{T_0}.
\end{align*}
To deal with the first term $E_1(t,x)$, we employ the triangle inequality and Propositions \ref{prop-auxilary-functions-1}-\ref{prop-auxilary-functions-3} to show
\begin{align*}
    &\chi_{\intt}(\xi)|\widehat{\phi}-\widehat{\ml{G}}_0\widehat{\phi}_1-\widehat{\ml{G}}_1\widehat{\phi}_0-(\widehat{\ml{H}}_0+\widehat{\ml{G}}_2)\widehat{\phi}_1-\widehat{\ml{G}}_3\widehat{T}_0|\\ &\qquad\lesssim\chi_{\intt}(\xi)\left|\widehat{\phi}-\widehat{J}_0-\widehat{J}_1\right|+ \chi_{\intt}(\xi)\left|\widehat{J}_0-(\widehat{\ml{G}}_0+\widehat{\ml{H}}_0)\widehat{\phi}_1\right|+\chi_{\intt}(\xi)\left|\widehat{J}_1-\widehat{\ml{G}}_1\widehat{\phi}_0-\widehat{\ml{G}}_2\widehat{\phi}_1-\widehat{\ml{G}}_3\widehat{T}_0\right|\\&\qquad\lesssim\chi_{\intt}(\xi)|\xi|\mathrm{e}^{-c|\xi|^2t}\left(|\widehat{\phi}_0|+|\widehat{\phi}_1|+|\widehat{T}_0|\right).
\end{align*}
Let us split the goal into two parts and using \eqref{Est-ext}, \eqref{Est-bdd} to get
\begin{align*}
    \|E_1(t,\cdot)\|_{L^2}&\lesssim\|\chi_{\intt}(\xi)\widehat{E}_1(t,\xi)\|_{L^2}+\big\|\big(1-\chi_{\intt}(\xi)\big)\widehat{E}_1(t,\xi)\big\|_{L^2}\\
    &\lesssim\big\|\chi_{\intt}(\xi)|\xi|\mathrm{e}^{-c|\xi|^2t}\big(|\widehat{\phi}_0(\xi)|+|\widehat{\phi}_1(\xi)|+|\widehat{T}_0(\xi)|\big)\big\|_{L^2}+\mathrm{e}^{-ct}\|(\phi_0,\phi_1,T_0)\|_{(L^2)^3}\\
    &\lesssim t^{-\frac{1}{2}-\frac{n}{4}}\left\|\left(\phi_0,\phi_1,T_0\right)\right\|_{(L^2\cap L^1)^3}
\end{align*}
for large time $t\gg1$. Next, we separate the second error term $E_2(t,x)$ into three parts
\begin{align*}
    E_2(t,x)&=\int_{|y|\leqslant t^{\frac{1}{8}}}\big(\ml{G}_0(t,x-y)-\ml{G}_0(t,x)+y\circ \nabla\ml{G}_0(t,x)\big)\phi_1(y)\mathrm{d}y\\
    &\quad+\int_{|y|\geqslant t^{\frac{1}{8}}}\big(\ml{G}_0(t,x-y)-\ml{G}_0(t,x)\big)\phi_1(y)\mathrm{d}y+\int_{|y|\geqslant t^{\frac{1}{8}}}y\circ \nabla\ml{G}_0(t,x)\phi_1(y)\mathrm{d}y.
\end{align*}
With the analogous manners to those in \eqref{Chen-Mod-04}, through the inequality 
\begin{align*}
    |\ml{G}_0(t,x-y)-\ml{G}_0(t,x)+y\circ \nabla\ml{G}_0(t,x)|\lesssim|y|^2|\nabla^2\ml{G}_0(t,x-\theta_1y)|\ \ \mbox{with}\ \ \theta_1\in(0,1),
\end{align*}
then we get for $t\gg1$ the following result:
\begin{align*}
    \|E_2(t,\cdot)\|_{L^2}&\lesssim t^{\frac{1}{4}}\|\ml{G}_0(t,\cdot)\|_{\dot{H}^2}\|\phi_1\|_{L^1}+\|\ml{G}_0(t,\cdot)\|_{\dot{H}^1}\|\phi_1\|_{L^{1,1}(|x|\geqslant t^{\frac{1}{8}})}= o(t^{-\frac{n}{4}}),
\end{align*}
thanks to our assumption $\phi_1\in L^{1,1}$. By the same way as those for \eqref{Chen-Mod-04}, we are able to derive
\begin{align*}
\|E_3(t,\cdot)\|_{L^2}+\|E_4(t,\cdot)\|_{L^2}+\|E_5(t,\cdot)\|_{L^2}=o(t^{-\frac{n}{4}})
\end{align*}
for large time since $\phi_0,\phi_1,T_0\in L^1$. Summarizing all derived estimates in the above, it leads to 
\begin{align}\label{Chen-Mod-07}
    \|\phi(t,\cdot)-\varphi(t,\cdot)-\psi(t,\cdot)\|_{L^2}\lesssim\sum\limits_{k=1,\dots,5}\|E_k(t,\cdot)\|_{L^2}=o(t^{-\frac{n}{4}})
\end{align}
for large time $t\gg1$, with the assumptions $\phi_0,\phi_1,T_0\in L^2\cap L^1$ as well as $\phi_1\in L^{1,1}$.\medskip

\noindent\textbf{\underline{Step 3. Lower bound estimates for the second asymptotic profile}}: Except the further error estimate \eqref{Chen-Mod-07}, we will estimate the second asymptotic profile $\psi(t,\cdot)$ in the $L^2$ norm. The partial Fourier transform of the second profile can be represented via
\begin{align*}
    \widehat{\psi}(t,\xi)&=-i(\xi\circ M_{\phi_1})\widehat{\ml{G}}_0(t,|\xi|)+\widehat{\ml{G}}_1(t,|\xi|)P_{\phi_0}+\left(\widehat{\ml{H}}_0(t,|\xi|)+\widehat{\ml{G}}_2(t,|\xi|)\right)P_{\phi_1}+\widehat{\ml{G}}_3(t,|\xi|)P_{T_0}\\
    &=:i\widehat{E}_6(t,|\xi|)+\widehat{E}_7(t,|\xi|).
\end{align*}
Because $\widehat{E}_6(t,|\xi|)\in\mb{R}$ and $\widehat{E}_7(t,|\xi|)\in\mb{R}$ are the imaginary part and the real part of $\widehat{\psi}(t,\xi)$, respectively, we may obtain  
\begin{align}\label{Chen-Mod-08}
    \|\widehat{\psi}(t,\xi)\|_{L^2}^2=\|\widehat{E}_6(t,|\xi|)\|_{L^2}^2+\|\widehat{E}_7(t,|\xi|)\|_{L^2}^2\geqslant\|\widehat{E}_6(t,|\xi|)\|_{L^2}^2 .
\end{align}
We apply polar coordinates to have
\begin{align*}
    \|\widehat{E}_6(t,|\xi|)\|_{L^2}^2&=\frac{1}{\gamma c_0^2}\int_0^{\infty}\int_{\mb{S}^{n-1}}|\sin(\sqrt{\gamma}c_0rt)|^2 \mathrm{e}^{-2\Gamma_0 r^2t}r^{n-1}(\omega\circ M_{\phi_1})^2\mathrm{d}\sigma_{\omega}\mathrm{d}r\\&=\frac{1}{\gamma c_0^2}\int_{\mb{S}^{n-1}}(\omega\circ M_{\phi_1})^2\mathrm{d}\sigma_{\omega}\int_0^{\infty}|\sin(\sqrt{\gamma}c_0rt)|^2 \mathrm{e}^{-2\Gamma_0 r^2t}r^{n-1}\mathrm{d}r.
\end{align*}
Then, by using the fact that 
\begin{align*}
    \int_{\mb{S}^{n-1}}\omega_j\omega_k\mathrm{d}\sigma_{\omega}=\left\{ \begin{array}{cl}
     \displaystyle{\frac{1}{n}|\mb{S}^{n-1}|} & \text{when} \ \ j=k, \\[0.5em]
     0 & \text{when}\ \ j\neq k,
\end{array} \right.
\end{align*}
we can calculate 
\begin{align*}
    \int_{\mb{S}^{n-1}}(\omega\circ M_{\phi_1})^2\mathrm{d}\sigma_{\omega}=\frac{|\mb{S}^{n-1}|}{n}|M_{\phi_1}|^2.
\end{align*}
Next, we compute the integral through the change of variables $\tau=r\sqrt{t}$ as well as $\eta=\sqrt{2\Gamma_0}\tau$, and using the Riemann-Lebesgue theorem for $t\gg1$ as follows:
\begin{align*}
    \int_0^{\infty}|\sin(\sqrt{\gamma}c_0rt)|^2 \mathrm{e}^{-2\Gamma_0 r^2t}r^{n-1}\mathrm{d}r
    &=t^{-\frac{n}{2}}\int_0^{\infty}|\sin(\sqrt{\gamma}c_0\tau\sqrt{t})|^2 \mathrm{e}^{-2\Gamma_0 \tau^2}\tau^{n-1}\mathrm{d}\tau\\
    &=\frac{1}{2}t^{-\frac{n}{2}}\left(\int_0^{\infty}\mathrm{e}^{-2\Gamma_0 \tau^2}\tau^{n-1}\mathrm{d}\tau-\int_0^{\infty} \cos(2\sqrt{\gamma}c_0\tau\sqrt{t}) \, \mathrm{e}^{-2\Gamma_0 \tau^2}\tau^{n-1}\mathrm{d}\tau\right)\\
    &=\frac{1}{2}t^{-\frac{n}{2}}\left((2\Gamma_0)^{-\frac{n}{2}}\int_0^{\infty}\mathrm{e}^{-\eta^2}\eta^{n-1}\mathrm{d}\eta+o(1)\right)\\
    &=\frac{1}{4}t^{-\frac{n}{2}}\left((2\Gamma_0)^{-\frac{n}{2}}\widetilde{\Gamma}\left(\frac{n}{2}\right)+o(1)\right),
\end{align*}
where $\widetilde{\Gamma}(z)$ is the well-known Gamma function. It shows that 
\begin{align}\label{Chen-Mod-09}
    \|\widehat{E}_6(t,|\xi|)\|_{L^2}^2=\frac{|\mb{S}^{n-1}|}{4n\gamma c_0^2}|M_{\phi_1}|^2t^{-\frac{n}{2}}\left[(2\Gamma_0)^{-\frac{n}{2}}\widetilde{\Gamma}\left(\frac{n}{2}\right)+o(1)\right]\gtrsim t^{-\frac{n}{2}}|M_{\phi_1}|^2
\end{align}
for large time $t\gg1$ provided that $|M_{\phi_1}|\neq0$.\medskip

\noindent \textbf{\underline{Step 4. Lower bound estimates for the error term}}: Let us summarize the estimates \eqref{Chen-Mod-07}-\eqref{Chen-Mod-09} and apply Minkowski's inequality to derive
\begin{align*}
   \|\phi(t,\cdot)-\varphi(t,\cdot)\|_{L^2}&\gtrsim\|\psi(t,\cdot)\|_{L^2}-\|\phi(t,\cdot)-\varphi(t,\cdot)-\psi(t,\cdot)\|_{L^2}\gtrsim t^{-\frac{n}{4}}|M_{\phi_1}|,
\end{align*}
 which completes the proof for the lower bound estimates.
\begin{remark}
According to the representation of $\widehat{E}_7(t,|\xi|)$, we may following the same approaches of the lower bound estimates for $\widehat{E}_6(t,|\xi|)$ to derive
\begin{align*}
\|\widehat{E}_7(t,|\xi|)\|_{L^2}^2\gtrsim t^{-\frac{n}{2}}\left(|P_{\phi_0}|^2+|P_{\phi_1}|^2+|P_{T_0}|^2\right)
\end{align*}
for large time $t\gg1$, by straightforward but tedious computations. Due to the main purpose of getting non-trivial sharp lower bound $t^{-\frac{n}{4}}$, the previous estimate is beyond the scope of this paper. 
\end{remark}

\section{Global (in time) inviscid limits for solutions}\label{Global-inviscid-limit}
$\ \ \ \ $We consider the limiting process of solutions in $L^{\infty}([0,\infty)\times\mb{R}^n)$ from the thermoviscous acoustic systems \eqref{Thermo-Acoustic-System} to the thermoelastic acoustic systems \eqref{Thermo-Acoustic-System-zero} as the momentum diffusion constant $\nu_0\downarrow0$, namely, the global (in time) inviscid limits defined in Section \ref{Section-Introduction}. To prove this convergence result, we will derive some uniform estimates for the inhomogeneous model in the phase space by suitable energy methods. Then, via preparing some bounded estimates for the limit model, we prove the convergences of the energy term and the velocity potential, respectively.  Moreover, the convergence rate demonstrated in this section coincides with the one in the WKB expansions of solution showing in Appendix \ref{Appendix-WKB-expansions}.

\subsection{Uniform estimates for the inhomogeneous model}\label{uniform-estimates-inho-equ}
$\ \ \ \ $As a preparation, motivated by the third order (in time) model \eqref{Thermo-Acoustic-System-third-equ}, according to the reduction methodology, we will establish convergence results for the corresponding third order (in time) evolution equations reduced by the acoustic systems \eqref{Thermo-Acoustic-System} and \eqref{Thermo-Acoustic-System-zero}, which may overcome some difficulties from the classical energy method. Particularly, we can derive refined estimates for the acoustic velocity potential which does not be contained in the classical energy of coupled systems. For this reason, we next consider the corresponding inhomogeneous model with vanishing first and second data as follows:
\begin{align}\label{inhomogeneous-model-equ}
\begin{cases}
\widehat{u}_{ttt}+[\gamma D_{\Th}+(1+\beta)\nu_0]|\xi|^2\widehat{u}_{tt}&\\
\quad\ \ +[\gamma c_0^2|\xi|^2+(1+\beta)\nu_0\gamma D_{\Th}|\xi|^4]\widehat{u}_t+
\gamma D_{\Th}c_0^2|\xi|^4\widehat{u}=\widehat{f},&\xi\in\mb{R}^n,\ t>0,\\
\widehat{u}(0,\xi)=0,\ \widehat{u}_t(0,\xi)=0,\ \widehat{u}_{tt}(0,\xi)=\widehat{u}_2(\xi),&\xi\in\mb{R}^n,
\end{cases}
\end{align}
where $\widehat{f}=\widehat{f}(t,\xi)$ is a flexible function as a source term to be determined in different situations, and $\widehat{u}_2=\widehat{u}_2(\xi)$ is assumed to be non-trivial data.

\begin{prop}\label{energy-est-1}
    The solution $\widehat{u}=\widehat{u}(t,\xi)$ to the Cauchy problem \eqref{inhomogeneous-model-equ} fulfills the next estimates:
    \begin{align*}
        \frac{1}{2}\frac{\mathrm{d}}{\mathrm{d}t}\left(\widehat{\mathscr{E}}_1+\widehat{\mathscr{E}}_2+\widehat{\mathscr{E}}_3\right)\leqslant \frac{2\gamma D_{\Th}c_0^2+k_1(1+\beta)\nu_0}{4(1+\beta)\nu_0\gamma D_{\Th}c_0^2|\xi|^2}|\widehat{f}|^2,
    \end{align*}
    where 
    \begin{align*}
        \widehat{\mathscr{E}}_1:&=\left|\widehat{u}_{tt}+\gamma D_{\Th}|\xi|^2\widehat{u}_t+k_1|\xi|^2\widehat{u}\right|^2,\\
   \widehat{\mathscr{E}}_2:&=k_1(\gamma c_0^2-k_1)|\xi|^4\left|\widehat{u}+\frac{k_1(1+\beta)\nu_0+\gamma D_{\Th}c_0^2}{\gamma^2D_{\Th}^2c_0^2|\xi|^2+k_1\left(\gamma c_0^2-k_1+(1+\beta)\nu_0\gamma D_{\Th}|\xi|^2\right)}\widehat{u}_t\right|^2,\\
    \widehat{\mathscr{E}}_3:&=\frac{k_1(\gamma c_0^2-k_1)^2|\xi|^2}{\left(\gamma^2D_{\Th}^2c_0^2+k_1(1+\beta)\nu_0\gamma D_{\Th}\right)|\xi|^2+k_1\left(\gamma c_0^2-k_1\right)}|\widehat{u}_t|^2,
    \end{align*}
    where $k_1\in(0,(\gamma-1)c_0^2)$ is a positive constant independent of $\nu_0$.
\end{prop}
\begin{proof}
    By proceeding \eqref{inhomogeneous-model-equ}$_1+\gamma D_{\Th}|\xi|^2\times\widehat{u}_{tt}+k_1|\xi|^2\times\widehat{u}_t$, we get
    \begin{align*}
    &\widehat{u}_{ttt}+\gamma D_{\Th}|\xi|^2\widehat{u}_{tt}+k_1|\xi|^2\widehat{u}_t\\
    &\qquad\qquad=-(1+\beta)\nu_0|\xi|^2\widehat{u}_{tt}-\left(\gamma c_0^2|\xi|^2-k_1|\xi|^2+(1+\beta)\nu_0\gamma D_{\Th}|\xi|^4\right)\widehat{u}_t-
    \gamma D_{\Th}c_0^2|\xi|^4\widehat{u}+\widehat{f}.
    \end{align*}
    Multiplying the above equation by the conjugate of $\widehat{u}_{tt}+\gamma D_{\Th}|\xi|^2\widehat{u}_t+k_1|\xi|^2\widehat{u}$, and taking the real part of the resultant equation, one arrives at
    \begin{align}\label{energy-equ}
    &\frac{1}{2}\frac{\mathrm{d}}{\mathrm{d}t}\left(\left|\widehat{u}_{tt}+\gamma D_{\Th}|\xi|^2\widehat{u}_t+k_1|\xi|^2\widehat{u}\right|^2\right)\notag\\
    &\qquad=-(1+\beta)\nu_0|\xi|^2|\widehat{u}_{tt}|^2-\gamma D_{\Th}\left(\gamma c_0^2-k_1+(1+\beta)\nu_0\gamma D_{\Th}|\xi|^2\right)|\xi|^4|\widehat{u}_t|^2-k_1\gamma D_{\Th}c_0^2|\xi|^6|\widehat{u}|^2\notag\\
    &\qquad\quad-\left[\gamma^2D_{\Th}^2c_0^2|\xi|^2+k_1\left(\gamma c_0^2-k_1+(1+\beta)\nu_0\gamma D_{\Th}|\xi|^2\right)\right]|\xi|^4\Re(\bar{\widehat{u}}\widehat{u}_t)\notag\\
    &\qquad\quad-\left(k_1(1+\beta)\nu_0+\gamma D_{\Th}c_0^2\right)|\xi|^4\Re(\bar{\widehat{u}}\widehat{u}_{tt})-\left(\gamma c_0^2-k_1+2(1+\beta)\nu_0\gamma D_{\Th}|\xi|^2\right)|\xi|^2\Re(\bar{\widehat{u}}_t\widehat{u}_{tt})\notag\\
    &\qquad\quad+\Re\left(\widehat{f}(\bar{\widehat{u}}_{tt}+\gamma D_{\Th}|\xi|^2\bar{\widehat{u}}_t+k_1|\xi|^2\bar{\widehat{u}})\right).
    \end{align}
    Direct computation shows that
    \begin{align}\label{energy-k2k3}
    \frac{1}{2}\frac{\mathrm{d}}{\mathrm{d}t}\left(k_2\left|\widehat{u}+k_3\widehat{u}_t\right|^2\right)=k_2\Re(\bar{\widehat{u}}\widehat{u}_t)+k_2k_3|\widehat{u}_t|^2+k_2k_3\Re(\bar{\widehat{u}}\widehat{u}_{tt})+k_2k_3^2\Re(\bar{\widehat{u}}_t\widehat{u}_{tt}).
    \end{align}
    In order to compensate $\Re(\bar{\widehat{u}}\widehat{u}_t)$ and $\Re(\bar{\widehat{u}}\widehat{u}_{tt})$ in \eqref{energy-equ}, we choose the positive real parameters $k_2$ and $k_3$ satisfying
    \begin{align*}
    k_2&=\left[\gamma^2D_{\Th}^2c_0^2|\xi|^2+k_1\left(\gamma c_0^2-k_1+(1+\beta)\nu_0\gamma D_{\Th}|\xi|^2\right)\right]|\xi|^4,\\
    k_3&=\frac{k_1(1+\beta)\nu_0+\gamma D_{\Th}c_0^2}{\gamma^2D_{\Th}^2c_0^2|\xi|^2+k_1\left[\gamma c_0^2-k_1+(1+\beta)\nu_0\gamma D_{\Th}|\xi|^2\right]},
    \end{align*}
    so that $k_2k_3=\left[k_1(1+\beta)\nu_0+\gamma D_{\Th}c_0^2\right]|\xi|^4$.
    Moreover, it holds that
    \begin{align*}
    &\frac{1}{2}\frac{\mathrm{d}}{\mathrm{d}t}\big[\left(\left(\gamma c_0^2-k_1+2(1+\beta)\nu_0\gamma D_{\Th}|\xi|^2\right)|\xi|^2-k_2k_3^2\right)|\widehat{u}_t|^2\big]\\
     &\qquad\qquad=\left[\left(\gamma c_0^2-k_1+2(1+\beta)\nu_0\gamma D_{\Th}|\xi|^2\right)|\xi|^2-k_2k_3^2\right]\Re(\bar{\widehat{u}}_t\widehat{u}_{tt}).
    \end{align*}
    Therefore, we summarize the derived equalities to see
    \begin{align}\label{Add-Wenhui}
    \frac{1}{2}\frac{\mathrm{d}}{\mathrm{d}t}&\left[\widehat{\mathscr{E}}_1+k_2\left|\widehat{u}+k_3\widehat{u}_t\right|^2+\left(\left(\gamma c_0^2-k_1+2(1+\beta)\nu_0\gamma D_{\Th}|\xi|^2\right)|\xi|^2-k_2k_3^2\right)|\widehat{u}_t|^2\right]\notag\\
    &=-(1+\beta)\nu_0|\xi|^2|\widehat{u}_{tt}|^2-\left[\gamma D_{\Th}\left(\gamma c_0^2-k_1+(1+\beta)\nu_0\gamma D_{\Th}|\xi|^2\right)|\xi|^4-k_2k_3\right]|\widehat{u}_t|^2\notag\\
    &\quad-k_1\gamma D_{\Th}c_0^2|\xi|^6|\widehat{u}|^2+\Re\left(\widehat{f}(\bar{\widehat{u}}_{tt}+\gamma D_{\Th}|\xi|^2\bar{\widehat{u}}_t+k_1|\xi|^2\bar{\widehat{u}})\right).
    \end{align}
    Then, employing Cauchy's inequality we get
    \begin{align}\label{Re-f}
        \Re\left(\widehat{f}(\bar{\widehat{u}}_{tt}+\gamma D_{\Th}|\xi|^2\bar{\widehat{u}}_t+k_1|\xi|^2\bar{\widehat{u}})\right)&\leqslant (1+\beta)\nu_0|\xi|^2|\widehat{u}_{tt}|^2+\frac{1}{4(1+\beta)\nu_0|\xi|^2}|\widehat{f}|^2\notag\\
        &\quad+(1+\beta)\nu_0\gamma^2D_{\Th}^2|\xi|^6|\widehat{u}_t|^2+\frac{1}{4(1+\beta)\nu_0|\xi|^2}|\widehat{f}|^2\notag\\
        &\quad+k_1\gamma D_{\Th}c_0^2|\xi|^6|\widehat{u}|^2+\frac{k_1}{4\gamma D_{\Th}c_0^2|\xi|^2}|\widehat{f}|^2.
    \end{align}
    Namely, we arrive at
    \begin{align*}
       \mbox{LHS of }\eqref{Add-Wenhui}
         &\leqslant -\left(\gamma^2D_{\Th}c_0^2|\xi|^4-\gamma D_{\Th}k_1|\xi|^4-k_2k_3\right)|\widehat{u}_t|^2+\left(\frac{1}{2(1+\beta)\nu_0|\xi|^2}+\frac{k_1}{4\gamma D_{\Th}c_0^2|\xi|^2}\right)|\widehat{f}|^2\\
         &\leqslant \frac{2\gamma D_{\Th}c_0^2+k_1(1+\beta)\nu_0}{4(1+\beta)\nu_0\gamma D_{\Th}c_0^2|\xi|^2}|\widehat{f}|^2,
    \end{align*}
 because we noticed that for $0<\nu_0\ll1$, it holds
    \begin{align*}
        \gamma^2D_{\Th}c_0^2|\xi|^4-\gamma D_{\Th}k_1|\xi|^4-k_2k_3=\gamma D_{\Th}|\xi|^4\left[(\gamma-1)c_0^2-k_1\right]-k_1(1+\beta)\nu_0|\xi|^4\geqslant 0.
    \end{align*}
    
    Let us simplify the left-hand side of \eqref{Add-Wenhui} via $k_2\geqslant k_1(\gamma c_0^2-k_1)|\xi|^4$. Moreover, we denote the coefficient of $|\widehat{u}_t|^2$ on the left-hand side of \eqref{Add-Wenhui} as 
    \begin{align*}
        \left(\gamma c_0^2-k_1+2(1+\beta)\nu_0\gamma D_{\Th}|\xi|^2\right)|\xi|^2-k_2k_3^2=:\frac{|\xi|^2k_4}{\left(\gamma^2D_{\Th}^2c_0^2+k_1(1+\beta)\nu_0\gamma D_{\Th}\right)|\xi|^2+k_1\left(\gamma c_0^2-k_1\right)},
    \end{align*}
    where the suitable parameter $k_4$ fulfills $k_4\geqslant k_1(\gamma c_0^2-k_1)^2$ since $k_1<(\gamma-1)c_0^2$ and $0<\nu_0\ll 1$. Summarizing the above we complete the proof of this proposition.
\end{proof}

\begin{prop}\label{energy-est-2}
    The solution $\widehat{u}=\widehat{u}(t,\xi)$ to the Cauchy problem \eqref{inhomogeneous-model-equ} fulfills the next estimates:
    \begin{align*}
        \frac{1}{2}\frac{\mathrm{d}}{\mathrm{d}t}\left(\widehat{\mathscr{E}}_1+\widehat{\mathscr{E}}_4+\widehat{\mathscr{E}}_5\right)\leqslant \frac{2\gamma D_{\Th}c_0^2+k_1(1+\beta)\nu_0}{4(1+\beta)\nu_0\gamma D_{\Th}c_0^2|\xi|^2}|\widehat{f}|^2,
    \end{align*}
    where
    \begin{align*}
        \widehat{\mathscr{E}}_4:&=\frac{\left[k_1(1+\beta)\nu_0+\gamma D_{\Th}c_0^2\right]^2|\xi|^6}{\gamma c_0^2-k_1+2(1+\beta)\nu_0\gamma D_{\Th}|\xi|^2}\left|\widehat{u}+\frac{\gamma c_0^2-k_1+2(1+\beta)\nu_0\gamma D_{\Th}|\xi|^2}{\left[k_1(1+\beta)\nu_0+\gamma D_{\Th}c_0^2\right]|\xi|^2}\widehat{u}_t\right|^2\\
        \widehat{\mathscr{E}}_5:&=\frac{k_1\left(\gamma c_0^2-k_1\right)^2|\xi|^4}{\gamma c_0^2-k_1+2(1+\beta)\nu_0\gamma D_{\Th}|\xi|^2}|\widehat{u}|^2,
    \end{align*}
    where $k_1\in(0,(\gamma-1)c_0^2)$ is a positive constant independent of $\nu_0$.
\end{prop}
\begin{proof}
    Considering the same form as \eqref{energy-k2k3} with different parameters, namely,
    \begin{align*}
    \frac{1}{2}\frac{\mathrm{d}}{\mathrm{d}t}\left(k_5\left|\widehat{u}+k_6\widehat{u}_t\right|^2\right)=k_5\Re(\bar{\widehat{u}}\widehat{u}_t)+k_5k_6\left|\widehat{u}_t\right|^2+k_5k_6\Re(\bar{\widehat{u}}\widehat{u}_{tt})+k_5k_6^2\Re(\bar{\widehat{u}}_t\widehat{u}_{tt}),
    \end{align*}
    we choose the positive parameters 
    \begin{align*}
    k_5=\frac{\left[k_1(1+\beta)\nu_0+\gamma D_{\Th}c_0^2\right]^2|\xi|^6}{\gamma c_0^2-k_1+2(1+\beta)\nu_0\gamma D_{\Th}|\xi|^2}\ \ \mbox{and}\ \ k_6=\frac{\gamma c_0^2-k_1+2(1+\beta)\nu_0\gamma D_{\Th}|\xi|^2}{\left[k_1(1+\beta)\nu_0+\gamma D_{\Th}c_0^2\right]|\xi|^2},
    \end{align*}
    so that 
    \begin{align*}
    k_5k_6=\left[k_1(1+\beta)\nu_0+\gamma D_{\Th}c_0^2\right]|\xi|^4\ \ \mbox{as well as}\ \  k_5k_6^2=\left[\gamma c_0^2-k_1+2(1+\beta)\nu_0\gamma D_{\Th}|\xi|^2\right]|\xi|^2,
    \end{align*}
    which compensate $\Re(\bar{\widehat{u}}\widehat{u}_{tt})$ and $\Re(\bar{\widehat{u}}_t\widehat{u}_{tt})$ in \eqref{energy-equ}.
    Then, we obtain 
    \begin{align*}
        &\frac{1}{2}\frac{\mathrm{d}}{\mathrm{d}t}\left[\left|\widehat{u}_{tt}+\gamma D_{\Th}|\xi|^2\widehat{u}_t+k_1|\xi|^2\widehat{u}\right|^2+k_5\left|\widehat{u}+k_6\widehat{u}_t\right|^2\right]\\
        &\quad=-(1+\beta)\nu_0|\xi|^2|\widehat{u}_{tt}|^2-k_1\gamma D_{\Th}c_0^2|\xi|^6|\widehat{u}|^2-k_7|\xi|^4\Re(\bar{\widehat{u}}\widehat{u}_t)+\Re\left(\widehat{f}(\bar{\widehat{u}}_{tt}+\gamma D_{\Th}|\xi|^2\bar{\widehat{u}}_t+k_1|\xi|^2\bar{\widehat{u}})\right)\\
        &\quad\quad-\left[\gamma D_{\Th}\left(c_0^2\gamma-k_1+(1+\beta)\nu_0\gamma D_{\Th}|\xi|^2\right)|\xi|^4-\left(k_1(1+\beta)\nu_0+\gamma D_{\Th}c_0^2\right)|\xi|^4\right]|\widehat{u}_t|^2,
    \end{align*}
where the positive parameter $k_7$ is denoted by
    \begin{align*}
        k_7:=\gamma^2D_{\Th}^2c_0^2|\xi|^2+k_1\left(\gamma c_0^2-k_1+(1+\beta)\nu_0\gamma D_{\Th}|\xi|^2\right)-\frac{\left(k_1(1+\beta)\nu_0+\gamma D_{\Th}c_0^2\right)^2|\xi|^2}{\gamma c_0^2-k_1+2(1+\beta)\nu_0\gamma D_{\Th}|\xi|^2}.
    \end{align*}
    Similarly, we arrive at
    \begin{align*}
        \frac{1}{2}\frac{\mathrm{d}}{\mathrm{d}t}\left[\widehat{\mathscr{E}}_1+k_5\left|\widehat{u}+k_6\widehat{u}_t\right|^2+k_7|\xi|^4\left|\widehat{u}\right|^2\right]\leqslant \frac{2\gamma D_{\Th}c_0^2+k_1(1+\beta)\nu_0}{4(1+\beta)\nu_0\gamma D_{\Th}c_0^2|\xi|^2}|\widehat{f}|^2,
    \end{align*}
    by using \eqref{Re-f} and the relation
    \begin{align*}
        \frac{1}{2}\frac{\mathrm{d}}{\mathrm{d}t}\left(k_7|\xi|^4|\widehat{u}|^2\right)=k_7|\xi|^4\Re(\bar{\widehat{u}}\widehat{u}_t).
    \end{align*}
    According to the fact that
    \begin{align*}
        k_7\geqslant \frac{k_1\left(\gamma c_0^2-k_1\right)^2}{\gamma c_0^2-k_1+2(1+\beta)\nu_0\gamma D_{\Th}|\xi|^2},
    \end{align*}
   it simplifies the energy term, and we complete the proof.
\end{proof}

\subsection{Some estimates for the inviscid model}\label{Sub-Section-Elastic-Acoustic}
$\ \ \ \ $Let us give some estimates for the velocity potential of the inviscid thermoviscous acoustic systems \eqref{Thermo-Acoustic-System-zero}. Recalling the third order (in time) differential equation in the phase space for $\widehat{\phi}=\widehat{\phi}(t,\xi)$, we know that $\widehat{\phi}^{(0)}=\widehat{\phi}^{(0)}(t,\xi)$ satisfies the following differential equation:
\begin{align}\label{equations-zero-phase}
\begin{cases}
\widehat{\phi}_{ttt}^{(0)}+\gamma D_{\Th}|\xi|^2\widehat{\phi}_{tt}^{(0)}+\gamma c_0^2|\xi|^2\widehat{\phi}_t^{(0)}+
\gamma D_{\Th}c_0^2|\xi|^4\widehat{\phi}^{(0)}=0,&\xi\in\mb{R}^n,\ t>0,\\
\widehat{\phi}^{(0)}(0,\xi)=\widehat{\phi}_0(\xi),\   \widehat{\phi}_t^{(0)}(0,\xi)=\widehat{\phi}_1(\xi),&\xi\in\mb{R}^n,\\ 
\widehat{\phi}_{tt}^{(0)}(0,\xi)=-\gamma c_0^2|\xi|^2\widehat{\phi}_0(\xi)+\alpha_p c_0^2\gamma D_{\Th}|\xi|^2\widehat{T}_0(\xi),&\xi\in\mb{R}^n.
\end{cases}
\end{align}
  The corresponding characteristic equation to \eqref{equations-zero-phase} is given by
\begin{align}\label{characterictic-zero-equ}
    \zeta^3+\gamma D_{\Th}|\xi|^2\zeta^2+\gamma c_0^2|\xi|^2\zeta+\gamma D_{\Th}c_0^2|\xi|^4=0.  
\end{align}
Then,  we obtain the following proposition:
\begin{prop}
    The   roots $\zeta_j=\zeta_j(|\xi|)$ of the  cubic   \eqref{characterictic-zero-equ} can be expanded as follows.
    \begin{itemize}
        \item Concerning $\xi\in\ml{Z}_{\intt}(\varepsilon_0)$, three roots behave as 
        \begin{align*}
            \zeta_1(|\xi|)=-D_{\Th}|\xi|^2+\ml{O}(|\xi|^4),\quad\zeta_{2,3}(|\xi|)=\mp i\sqrt{\gamma}c_0|\xi|-\frac{(\gamma-1)D_{\Th}}{2}|\xi|^2+\ml{O}(|\xi|^3).
        \end{align*}
        \item Concerning $\xi\in\ml{Z}_{\extt}(N_0)$, three roots behave as 
        \begin{align*}
            \zeta_1(|\xi|)=-\gamma D_{\Th}|\xi|^2+\ml{O}(1),\quad\zeta_{2,3}(|\xi|)=\mp ic_0|\xi|-\frac{(\gamma-1)c_0^2}{2\gamma D_{\Th}}+\ml{O}(|\xi|^{-1}).
         \end{align*}
    \end{itemize}
\end{prop}
Since the discriminant of the characteristic equation \eqref{characterictic-zero-equ} are strictly negative for both $\xi\in\ml{Z}_{\intt}(\varepsilon_0)$ and $\xi\in\ml{Z}_{\extt}(N_0)$, the solution   owns the representation
\begin{align}\label{phi-0}
     \widehat{\phi}^{(0)}\notag&=\frac{(\gamma c_0^2|\xi|^2-\zeta_{\mathrm{I}}^2-\zeta_{\mathrm{R}}^2)\widehat{\phi}_0+2\zeta_{\mathrm{R}}\widehat{\phi}_1-\alpha_p c_0^2\gamma D_{\Th}|\xi|^2\widehat{T}_0}{2\zeta_{\mathrm{R}}\zeta_1-\zeta_{\mathrm{I}}^2-\zeta_{\mathrm{R}}^2-\zeta_1^2}\mathrm{e}^{\zeta_1t}\\ \notag
     &\quad+\frac{(2\zeta_{\mathrm{R}}\zeta_1-\zeta_1^2-\gamma c_0^2|\xi|^2)\widehat{\phi}_0-2\zeta_{\mathrm{R}}\widehat{\phi}_1+\alpha_p c_0^2\gamma D_{\Th}|\xi|^2\widehat{T}_0}{2\zeta_{\mathrm{R}}\zeta_1-\zeta_{\mathrm{I}}^2-\zeta_{\mathrm{R}}^2-\zeta_1^2}\cos(\zeta_{\mathrm{I}}t)\mathrm{e}^{\zeta_{\mathrm{R}}t}\\ \notag
     &\quad+\frac{[\zeta_1(\zeta_{\mathrm{R}}\zeta_1+\zeta_{\mathrm{I}}^2-\zeta_{\mathrm{R}}^2)+\gamma c_0^2|\xi|^2(\zeta_{\mathrm{R}}-\zeta_1)]\widehat{\phi}_0}{\zeta_{\mathrm{I}}(2\zeta_{\mathrm{R}}\zeta_1-\zeta_{\mathrm{I}}^2-\zeta_{\mathrm{R}}^2-\zeta_1^2)}\sin(\zeta_{\mathrm{I}}t)\mathrm{e}^{\zeta_{\mathrm{R}}t}\\ 
     &\quad+\frac{(\zeta_{\mathrm{R}}^2-\zeta_{\mathrm{I}}^2-\zeta_1^2)\widehat{\phi}_1-(\zeta_{\mathrm{R}}-\zeta_1)\alpha_p c_0^2\gamma D_{\Th}|\xi|^2\widehat{T}_0}{\zeta_{\mathrm{I}}(2\zeta_{\mathrm{R}}\zeta_1-\zeta_{\mathrm{I}}^2-\zeta_{\mathrm{R}}^2-\zeta_1^2)}\sin(\zeta_{\mathrm{I}}t)\mathrm{e}^{\zeta_{\mathrm{R}}t},
\end{align}
where $\zeta_{\mathrm{I}}=- \sqrt{\gamma}c_0|\xi|+\ml{O}(|\xi|^3)$, $\zeta_{\mathrm{R}}=-\frac{\gamma-1}{2}D_{\Th}|\xi|^2+\ml{O}(|\xi|^3)$ for $\xi\in\ml{Z}_{\intt}(\varepsilon_0)$, and $\zeta_{\mathrm{I}}=- c_0|\xi|+\ml{O}(|\xi|^{-1})$, $\zeta_{\mathrm{R}}=-\frac{(\gamma-1)c_0^2}{2\gamma D_{\Th}}+\ml{O}(|\xi|^{-1})$ for $\xi\in\ml{Z}_{\extt}(N_0)$. Thus, by plugging these roots into the representation, we are able to derive the following pointwise estimates in the phase space:
\begin{align*}
\chi_{\intt}(\xi)|\partial_t^j\widehat{\phi}^{(0)}|&\lesssim \chi_{\intt}(\xi)|\xi|^{j-1}\mathrm{e}^{-c|\xi|^2t}\left(|\xi|\,|\widehat{\phi}_0|+|\widehat{\phi}_1|+|\xi|\,|\widehat{T}_0|\right),\\
\big(1-\chi_{\intt}(\xi)\big)|\partial_t^j\widehat{\phi}^{(0)}|&\lesssim \big(1-\chi_{\intt}(\xi)\big)\mathrm{e}^{-ct}\times
\begin{cases}
\langle\xi\rangle^j|\widehat{\phi}_0|+\langle\xi\rangle^{j-1}|\widehat{\phi}_1|+\langle\xi\rangle^{j-1}|\widehat{T}_0|&\mbox{when}\ \ j=0,1,\\
\langle\xi\rangle^2|\widehat{\phi}_0|+\langle\xi\rangle|\widehat{\phi}_1|+\langle\xi\rangle^{2}|\widehat{T}_0|&\mbox{when}\ \ j=2.
\end{cases}
\end{align*}

\begin{prop}\label{prop-limit-solution-1}
    Suppose that initial data $\langle D\rangle^{s+j+1}\phi_0,\langle D\rangle^{s+j}\phi_1,\langle D\rangle^{s+j}T_0\in L^1$, with $s>n+2$ and $j=0,1,2$. Then, the solution to the Cauchy problem \eqref{equations-zero-phase} satisfies the following estimates:
    \begin{align*}
    \int_{\mb{R}^n}\left(\int_0^t\left||\xi|^j\widehat{\phi}_{tt}^{(0)}(\tau,\xi)\right|^2+\left||\xi|^{j+2}\widehat{\phi}_{t}^{(0)}(\tau,\xi)\right|^2\mathrm{d}\tau\right)^{\frac{1}{2}}\mathrm{d}\xi\leqslant C \left\|\langle D\rangle ^{s+j}\left(\langle D\rangle\phi_0,\phi_1,T_0\right)\right\|_{(L^1)^3}.
    \end{align*}
\end{prop}
\begin{proof}
    According to the pointwise estimates of $\widehat{\phi}_{tt}^{(0)}$ as well as $\widehat{\phi}_t^{(0)}$ in the phase space that is
    \begin{align*}
    	|\xi|^{2j}|\widehat{\phi}_{tt}^{(0)}|^2+|\xi|^{2j+4}|\widehat{\phi}_t^{(0)}|^2\lesssim\begin{cases}
    	|\xi|^{2j+2}\mathrm{e}^{-c|\xi|^2t}\left(|\widehat{\phi}_0|^2+|\widehat{\phi}_1|^2+|\widehat{T}_0|^2\right)&\mbox{when}\ \ |\xi|\leqslant 1,\\
    	\langle\xi\rangle^{2j+4}\mathrm{e}^{-ct}\left(\langle\xi\rangle^2|\widehat{\phi}_0|^2+|\widehat{\phi}_1|^2+|\widehat{T}_0|^2\right)&\mbox{when}\ \ |\xi|\geqslant 1,
    	\end{cases}
    \end{align*}
we divide the integral into two parts to get 
    \begin{align*}
        &\int_{\mb{R}^n}\left(\int_0^t\left||\xi|^j\widehat{\phi}_{tt}^{(0)}(\tau,\xi)\right|^2+\left||\xi|^{j+2}\widehat{\phi}_{t}^{(0)}(\tau,\xi)\right|^2\mathrm{d}\tau\right)^{\frac{1}{2}}\mathrm{d}\xi\\ 
        &\qquad\qquad\leqslant C\int_{|\xi|\leqslant 1}\left(\int_0^t\mathrm{e}^{-c|\xi|^2\tau}\mathrm{d}\tau\right)^{\frac{1}{2}}|\xi|^{j+1}\mathrm{d}\xi\left(\|\widehat{\phi}_0\|_{L^{\infty}}+\|\widehat{\phi}_1\|_{L^{\infty}}+\|\widehat{T}_0\|_{L^{\infty}}\right)\\
        &\qquad\qquad\quad+C\int_{|\xi|\geqslant 1}\left(\int_0^t \mathrm{e}^{-c\tau}\mathrm{d}\tau\right)^{\frac{1}{2}}\left(\langle\xi\rangle^{j+3}|\widehat{\phi}_0(\xi)|+\langle\xi\rangle^{j+2}|\widehat{\phi}_1(\xi)|+\langle\xi\rangle^{j+2}|\widehat{T}_0(\xi)|\right)\mathrm{d}\xi\\
         &\qquad\qquad\leqslant C\int_{|\xi|\leqslant 1}\left(1-\mathrm{e}^{-c|\xi|^2t}\right)^{\frac{1}{2}}|\xi|^j\mathrm{d}\xi\left(\|\widehat{\phi}_0\|_{L^{\infty}}+\|\widehat{\phi}_1\|_{L^{\infty}}+\|\widehat{T}_0\|_{L^{\infty}}\right)\\
        &\qquad\qquad\quad+C\int_{|\xi|\geqslant 1}\left\langle\xi\right\rangle^{2-s}\mathrm{d}\xi\left(\left\|\left\langle\xi\right\rangle^{s+j+1}\widehat{\phi}_0\right\|_{L^{\infty}}+\left\|\left\langle\xi\right\rangle^{s+j}\widehat{\phi}_1\right\|_{L^{\infty}}+\left\|\left\langle\xi\right\rangle^{s+j}\widehat{T}_0\right\|_{L^{\infty}}\right) \\
         &\qquad\qquad\leqslant C\left\|\left(\phi_0,\phi_1,T_0\right)\right\|_{(L^1)^3}+C\left(\left\|\langle D\rangle^{s+j+1}\phi_0\right\|_{L^1}+\left\|\langle D\rangle^{s+j}\phi_1\right\|_{L^1}+\left\|\langle D\rangle^{s+j}T_0\right\|_{L^1}\right),
    \end{align*}
where we employed the Hausdorff-Young inequality and the integrability for large frequencies when $s>n+2$.
\end{proof}
\begin{prop}\label{prop-limit-solution-2}
    Suppose that initial data $\langle D\rangle^{s}\phi_0,\langle D\rangle^{s-1}\phi_1,\langle D\rangle^{s-1}T_0\in L^1$, with $s>n+2$, and $n\geqslant 2$. Then, the solution  to the Cauchy problem \eqref{equations-zero-phase} satisfies the following estimate:
    \begin{align*}
    \int_{\mb{R}^n}\left(\int_0^t\left|\frac{1}{|\xi|}\widehat{\phi}_{tt}^{(0)}(\tau,\xi)\right|^2+\left||\xi|\widehat{\phi}_{t}^{(0)}(\tau,\xi)\right|^2\mathrm{d}\tau\right)^{\frac{1}{2}}\mathrm{d}\xi\leqslant C \left\|\langle D\rangle ^{s-1}\left(\langle D\rangle\phi_0,\phi_1,T_0\right)\right\|_{(L^1)^3}.
    \end{align*}
\end{prop}
\begin{proof}
	At this moment, we have the estimates in the phase space
	    \begin{align*}
		|\xi|^{-2}|\widehat{\phi}_{tt}^{(0)}|^2+|\xi|^{2}|\widehat{\phi}_t^{(0)}|^2\lesssim\begin{cases}
			\mathrm{e}^{-c|\xi|^2t}\left(|\widehat{\phi}_0|^2+|\widehat{\phi}_1|^2+|\widehat{T}_0|^2\right)&\mbox{when}\ \ |\xi|\leqslant 1,\\
			\langle\xi\rangle^{2}\mathrm{e}^{-ct}\left(\langle\xi\rangle^2|\widehat{\phi}_0|^2+|\widehat{\phi}_1|^2+|\widehat{T}_0|^2\right)&\mbox{when}\ \ |\xi|\geqslant 1.
		\end{cases}
	\end{align*}
By following the same approach as the one of Proposition \ref{prop-limit-solution-1}, we may complete the proof, where the restriction $n\geqslant 2$ comes from the singularity of small frequencies after an integration with respect to the time variable, that is,
\begin{align*}
\int_{|\xi|\leqslant 1}\left(\int_0^t\mathrm{e}^{-c|\xi|^2\tau}\mathrm{d}\tau\right)^{\frac{1}{2}}\mathrm{d}\xi\lesssim\int_{|\xi|\leqslant 1}|\xi|^{-1}\mathrm{d}\xi\lesssim \int_0^1r^{n-2}\mathrm{d}r<\infty 
\end{align*}
when $n\geqslant 2$.
\end{proof}
\begin{remark}
    By taking the additional assumptions $|P_{\phi_1}|=0$ and $\phi_1\in L^{1,1}$, we can obtain the bounded estimates
    \begin{align*}
        \int_{\mb{R}^n}\left(\int_0^t\left|\frac{1}{|\xi|}\widehat{\phi}_{tt}^{(0)}(\tau,\xi)\right|^2\mathrm{d}\tau\right)^{\frac{1}{2}}\mathrm{d}\xi\leqslant C\left\|\left(\phi_0,\phi_1,T_0\right)\right\|_{L^1\times L^{1,1}\times L^1}+C \left\|\langle D\rangle ^{s-2}\left(\langle D\rangle\phi_0,\phi_1,\langle D\rangle T_0\right)\right\|_{(L^1)^3},
    \end{align*}
    even for $n=1$ due to the fact $|\widehat{f}-P_f|\leqslant C|\xi|\|f\|_{L^{1,1}}$ from \cite{Ikehata=2014} dealing with the singularity $|\xi|^{-1}$ for small frequencies.
\end{remark}
\begin{prop}\label{phi-t-0}
    Suppose that initial data $\langle D\rangle^{s-1}\phi_0,\langle D\rangle^{s-2}\phi_1,\langle D\rangle^{s-2}T_0\in L^1$, with $s>n+2$. Then, the solution to the Cauchy problem \eqref{equations-zero-phase} satisfies the following estimate:
    \begin{align*}
    \sup_{t\in[0,\infty)}\int_{\mb{R}^n}|\widehat{\phi}_{t}^{(0)}(t,\xi)|\mathrm{d}\xi\leqslant C \left\|\langle D\rangle ^{s-2}\left(\langle D\rangle\phi_0,\phi_1,T_0\right)\right\|_{(L^1)^3}.
    \end{align*}
\end{prop}


\subsection{Inviscid limit for the energy terms: Proof of the Theorem \ref{Thm-Inviscid-Energy}}
$\ \ \ \ $Let us define the difference of solutions between the viscous model \eqref{Thermo-Acoustic-System} and the inviscid model \eqref{Thermo-Acoustic-System-zero} by $u=u(t,x)$ such that
\begin{align*}
u(t,x):=\phi(t,x)-\phi^{(0)}(t,x).
\end{align*}
Then, $\widehat{u}=\widehat{u}(t,\xi)$ is a solution to the inhomogeneous Cauchy problem \eqref{inhomogeneous-model-equ} with the source term
\begin{align*}
\widehat{f}=-(1+\beta)\nu_0|\xi|^2\widehat{\phi}_{tt}^{(0)}-(1+\beta)\nu_0\gamma D_{\Th}|\xi|^4\widehat{\phi}_{t}^{(0)},
\end{align*}
 and the third data $\widehat{u}_2=-(1+\beta)\nu_0|\xi|^2\widehat{\phi}_1$.
 
 We next apply the derived energy estimates in Subsection \ref{uniform-estimates-inho-equ} to obtain the convergence of the energy terms. The combination of Proposition \ref{energy-est-1}, namely,
 \begin{align*}
 	\frac{|\xi|^2}{|\xi|^2+1}|\widehat{u}_t(t,\xi)|^2\leqslant C|\widehat{u}_2(\xi)|^2+ \frac{C}{\nu_0|\xi|^2}\int_0^t|\widehat{f}(\tau,\xi)|^2\mathrm{d}\tau,
 \end{align*}
  and the Hausdorff-Young inequality yields
\begin{align}\label{u-t}
    &\notag\left\|u_t(t,\cdot)\right\|_{L^\infty}\leqslant \left\|\widehat{u}_t(t,\xi)\right\|_{L^1}\\
    &\notag\qquad\leqslant  C\nu_0\int_{\mb{R}^n}\langle\xi\rangle^2 |\widehat{\phi}_1(\xi)|\mathrm{d}\xi+C\sqrt{\nu_0}\int_{\mb{R}^n}\left(\int_0^t\left|(|\xi|+1)\widehat{\phi}_{tt}^{(0)}(\tau,\xi)\right|^2+\left|(|\xi|^3+|\xi|^2)\widehat{\phi}_{t}^{(0)}(\tau,\xi)\right|^2\mathrm{d}\tau\right)^{\frac{1}{2}}\mathrm{d}\xi\\
    &\notag\qquad\leqslant C\nu_0\int_{\mb{R}^n}\langle \xi\rangle^{1-s}\mathrm{d}\xi\, \|\langle\xi\rangle^{s+1}\widehat{\phi}_1\|_{L^{\infty}}+C\sqrt{\nu_0}\left(\left\|\langle D\rangle^{s+2}\phi_0\right\|_{L^1}+\left\|\langle D\rangle^{s+1}\phi_1\right\|_{L^1}+\left\|\langle D\rangle^{s+1}T_0\right\|_{L^1}\right)\\
    &\qquad\leqslant C\sqrt{\nu_0}\left(\left\|\langle D\rangle^{s+2}\phi_0\right\|_{L^1}+\left\|\langle D\rangle^{s+1}\phi_1\right\|_{L^1}+\left\|\langle D\rangle^{s+1}T_0\right\|_{L^1}\right),
\end{align}
where in the third line we used Proposition \ref{prop-limit-solution-1} when $j=0,1$ with $s>n+2$ and $n\geqslant 2$. Similarly, we may apply Proposition \ref{energy-est-2} and Proposition \ref{prop-limit-solution-1} when $j=1,2$ to get
\begin{align}\label{u-D2}
    \left\|\Delta u(t,\cdot)\right\|_{L^\infty}&\leqslant C\nu_0\int_{\mb{R}^n}(|\xi|^3+|\xi|^2) |\widehat{\phi}_1(\xi)|\mathrm{d}\xi\notag\\
    &\quad+C\sqrt{\nu_0}\int_{\mb{R}^n}\left(\int_0^t\left|(|\xi|^2+|\xi|)\widehat{\phi}_{tt}^{(0)}(\tau,\xi)\right|^2+\left|(|\xi|^4+|\xi|^3)\widehat{\phi}_{t}^{(0)}(\tau,\xi)\right|^2\mathrm{d}\tau\right)^{\frac{1}{2}}\mathrm{d}\xi\notag\\
    &\leqslant C\sqrt{\nu_0}\left(\left\|\langle D\rangle^{s+3}\phi_0\right\|_{L^1}+\left\|\langle D\rangle^{s+2}\phi_1\right\|_{L^1}+\left\|\langle D\rangle^{s+2}T_0\right\|_{L^1}\right).
\end{align}

Let us turn to the singular limits for the velocity potential. Employing the Propositions \ref{energy-est-2}, \ref{prop-limit-solution-1} and \ref{prop-limit-solution-2}, then we obtain
\begin{align}\label{50}
\left\| u(t,\cdot)\right\|_{L^\infty}
	&\leqslant C\nu_0\int_{\mb{R}^n}(|\xi|+1) |\widehat{\phi}_1(\xi)|\mathrm{d}\xi\notag\\
	&\quad+C\sqrt{\nu_0}\int_{\mb{R}^n}\left(\int_0^t\left|\left(1+\frac{1}{|\xi|}\right)\widehat{\phi}_{tt}^{(0)}(\tau,\xi)\right|^2+\left|(|\xi|^2+|\xi|)\widehat{\phi}_{t}^{(0)}(\tau,\xi)\right|^2\mathrm{d}\tau\right)^{\frac{1}{2}}\mathrm{d}\xi\notag\\
	& \leqslant C\nu_0\int_{\mb{R}^n}(|\xi|+1) |\widehat{\phi}_1(\xi)|\mathrm{d}\xi+C\sqrt{\nu_0}\left(\left\|\langle D\rangle^{s+1}\phi_0\right\|_{L^1}+\left\|\langle D\rangle^{s}\phi_1\right\|_{L^1}+\left\|\langle D\rangle^{s}T_0\right\|_{L^1}\right)\notag\\
	&\leqslant C\sqrt{\nu_0}\left\|\langle D\rangle ^{s}\left(\langle D\rangle\phi_0,\phi_1,T_0\right)\right\|_{(L^1)^3}.
\end{align}

Recalling the equation \eqref{Thermo-Acoustic-System-modified}, we may use the velocity potential to represent the temperature in the phase space
\begin{align*}
    \widehat{T}&=\frac{1}{\alpha_p c_0^2\gamma D_{\Th}|\xi|^2}\left(\widehat{\phi}_{tt}+\gamma c_0^2|\xi|^2\widehat{\phi}+(1+\beta)\nu_0|\xi|^2\widehat{\phi}_t\right),\\
    \widehat{T}^{(0)}&=\frac{1}{\alpha_p c_0^2\gamma D_{\Th}|\xi|^2}\left(\widehat{\phi}^{(0)}_{tt}+\gamma c_0^2|\xi|^2\widehat{\phi}^{(0)}\right).
\end{align*}
Thus, the difference of the temperature can be shown by
\begin{align*}
    \widehat{T}-\widehat{T}^{(0)}=\frac{1}{\alpha_p c_0^2\gamma D_{\Th}|\xi|^2}\left[\widehat{u}_{tt}+\gamma c_0^2|\xi|^2\widehat{u}+(1+\beta)\nu_0|\xi|^2\widehat{u}_t+(1+\beta)\nu_0|\xi|^2\widehat{\phi}^{(0)}_t\right].
\end{align*}
One obtains the $L^{\infty}$ norm for the difference by the triangle inequality as follows:
\begin{align*}
    \left\|T(t,\cdot)-T^{(0)}(t,\cdot)\right\|_{L^\infty}&\leqslant \left\|\widehat{T}(t,\xi)-\widehat{T}^{(0)}(t,\xi)\right\|_{L^1}\\
    &\leqslant C\int_{\mb{R}^n}\left|\frac{1}{|\xi|^2}\widehat{u}_{tt}+\widehat{u}+\nu_0\widehat{u}_t+\nu_0\widehat{\phi}^{(0)}_t\right|\mathrm{d}\xi\\
    &\leqslant C\int_{\mb{R}^n}\left(\frac{1}{|\xi|^2}\left|\widehat{u}_{tt}+\gamma D_{\Th}|\xi|^2\widehat{u}_t+k_1|\xi|^2\widehat{u}\right|+|\widehat{u}_t|+|\widehat{u}|+\nu_0|\widehat{\phi}^{(0)}_t|\right)\mathrm{d}\xi.
\end{align*}
Furthermore, we employ the pointwise estimates \eqref{u-t}, \eqref{50} and Proposition \ref{phi-t-0} to get
\begin{align*}
    \left\|T(t,\cdot)-T^{(0)}(t,\cdot)\right\|_{L^\infty}&\leqslant C\nu_0\int_{\mb{R}^n}|\widehat{\phi}_1(\xi)|\mathrm{d}\xi+C\sqrt{\nu_0}\int_{\mb{R}^n}\left(\int_0^t\left|\frac{1}{|\xi|}\widehat{\phi}_{tt}^{(0)}(\tau,\xi)\right|^2+\left||\xi|\widehat{\phi}_{t}^{(0)}(\tau,\xi)\right|^2\mathrm{d}\tau\right)^{\frac{1}{2}}\mathrm{d}\xi\\
    &\quad+C\sqrt{\nu_0} \left\|\langle D\rangle ^{s+1}\left(\langle D\rangle\phi_0,\phi_1,T_0\right)\right\|_{(L^1)^3}+C \nu_0\left\|\langle D\rangle ^{s-2}\left(\langle D\rangle\phi_0,\phi_1,T_0\right)\right\|_{(L^1)^3}\\
    &\leqslant C\sqrt{\nu_0}\left\|\langle D\rangle ^{s+1}\left(\langle D\rangle\phi_0,\phi_1,T_0\right)\right\|_{(L^1)^3},
\end{align*}
with $s>n+2$ and $n\geqslant 2$. Summarizing the derived estimates, we complete the proof.

\appendix
\section{WKB expansions with the small momentum diffusion coefficient}\label{Appendix-WKB-expansions}
$\ \ \ \ $In Section \ref{Global-inviscid-limit}, we have proven that the inviscid systems \eqref{Thermo-Acoustic-System-zero} is the limit model of the thermoviscious acoustic systems \eqref{Thermo-Acoustic-System}. We now use the multi-scale analysis to determine the expansions of solution with $0<\nu_0\ll1$ formally. Recalling that the coupled system \eqref{Thermo-Acoustic-System} can be reduced to the third order (in time) differential equation in the phase space \eqref{Thermo-Acoustic-System-third-equ}, then the velocity potential satisfies the following systems:
\begin{align}\label{Thermo-Acoustic-System-third-equ2}
\begin{cases}
\phi_{ttt}-[\gamma D_{\Th}+(1+\beta)\nu_0]\Delta\phi_{tt}&\\
\quad\ \ -\gamma c_0^2\Delta\phi_t+(1+\beta)\nu_0\gamma D_{\Th}\Delta^2\phi_t+
\gamma D_{\Th}c_0^2\Delta^2\phi=0,&x\in\mb{R}^n,\ t>0,\\
\phi(0,x)=\phi_0(x),\ \phi_t(0,x)=\phi_1(x),&x\in\mb{R}^n,\\
\phi_{tt}(0,x)=\gamma c_0^2\Delta\phi_0(x)+(1+\beta)\nu_0\Delta\phi_1(x)-\alpha_p c_0^2\gamma D_{\Th}\Delta T_0(x),&x\in\mb{R}^n.
\end{cases}
\end{align}
Based on the multi-scale method, the solution to \eqref{Thermo-Acoustic-System-third-equ2} has the following expansions:
\begin{align}\label{WKB-expansion}
    \phi(t,x)=\sum_{j\geqslant0}\sqrt{\nu_0}^j\left(\phi^{I,j}(t,x)+\phi^{L,j}(t,\sqrt{\nu_0}x)\right),
\end{align}
where each term in \eqref{WKB-expansion} is assumed to be smooth. Additionally, by taking $z=\sqrt{\nu_0}x$, we assume that the profiles of the remaining terms $\phi^{L,j}$ are decaying (or zero) when $z\to 0$. The expansion \eqref{WKB-expansion} should satisfy the initial conditions. Plugging \eqref{WKB-expansion} into \eqref{Thermo-Acoustic-System-third-equ2}, one notices that
\begin{align}\label{WKB-expression}
    0=\notag&\sum_{j\geqslant0}\nu_0^\frac{j}{2}\left(\partial_t^3\phi^{I,j}+\partial_t^3\phi^{L,j}\right)-\gamma D_{\Th}\sum_{j\geqslant0}\nu_0^\frac{j}{2}\left(\Delta\partial_t^2\phi^{I,j}+\nu_0\Delta_z\partial_t^2\phi^{L,j}\right)\\ \notag&-(1+\beta)\sum_{j\geqslant0}\nu^{\frac{j}{2}+1}\left(\Delta\partial_t^2\phi^{I,j}+\nu_0\Delta_z\partial_t^2\phi^{L,j}\right)-\gamma c_0^2\sum_{j\geqslant0}\nu_0^\frac{j}{2}\left(\Delta\partial_t\phi^{I,j}+\nu_0\Delta_z\partial_t\phi^{L,j}\right)\\ &+(1+\beta)\gamma D_{\Th}\sum_{j\geqslant0}\nu_0^{\frac{j}{2}+1}\left(\Delta^2\partial_t\phi^{I,j}+\nu_0^2\Delta_z^2\partial_t\phi^{L,j}\right) +\gamma D_{\Th} c_0^2\sum_{j\geqslant0}\nu_0^\frac{j}{2}\left(\Delta^2\phi^{I,j}+\nu_0^2\Delta_z^2\phi^{L,j}\right),
\end{align}
where $\phi^{I,j}=\phi^{I,j}(t,x)$, $\phi^{L,j}=\phi^{L,j}(t,z)$ and the operator $\Delta_z:=\sum_{k=1}^n\partial_{z_k}^2$.  Concerning initial conditions with $k=0,1,2$, we obtain the next relations:
\begin{align*}
    \phi_k(x)=\partial_t^k\phi^{I,0}(0,x)+\partial_t^k\phi^{L,0}(0,z)+\sum_{j\geqslant1}\nu_0^\frac{j}{2}\left(\partial_t^k\phi^{I,j}(0,x)+\partial_t^k\phi^{L,j}(0,z)\right).
\end{align*}
Accordingly, we naturally consider that
\begin{itemize}
    \item $\phi^{I,0}(0,x)=\phi_0(x)$, and $\phi^{L,0}(0,x)=\phi^{I,j}(0,x)=\phi^{L,j}(0,z)=0$ for $j\geqslant 1$;
    \item $\phi_t^{I,0}(0,x)=\phi_1(x)$, and $\phi_t^{L,0}(0,z)=\phi_t^{I,j}(0,x)=\phi_t^{L,j}(0,z)=0$ for $j\geqslant 1$;
    \item $\phi_{tt}^{I,0}(0,x)=\gamma c_0^2\Delta\phi_0(x)-\alpha_pc_0^2\gamma D_{\Th}\Delta T_0(x)$, $\phi_{tt}^{I,2}(0,x)+\phi_{tt}^{L,2}(0,z)=(1+\beta)\Delta\phi_1(x)$ and $\phi_{tt}^{L,0}(0,z)=\phi_{tt}^{I,j}(0,x)=\phi_{tt}^{L,j}(0,z)=0$ for $j=1$ and $j\geqslant 3$.
\end{itemize}

 First, let us collect the terms with $\ml{O}(1)$, that is,
\begin{align}\label{WKB-0}
\begin{cases}
\partial_t^3\phi^{I,0}-\gamma D_{\Th}\Delta\partial_t^2\phi^{I,0}-\gamma c_0^2\Delta\partial_t\phi^{I,0}+
\gamma D_{\Th}c_0^2\Delta^2\phi^{I,0}=-\partial_t^3\phi^{L,0},&x\in\mb{R}^n,\ t>0,\\
\phi^{I,0}(0,x)=\phi_0(x),\ \phi_t^{I,0}(0,x)=\phi_1(x),&x\in\mb{R}^n,\\ \phi_{tt}^{I,0}(0,x)=\gamma c_0^2\Delta\phi_0(x)-\alpha_p c_0^2\gamma D_{\Th}\Delta T_0(x),&x\in\mb{R}^n.
\end{cases}
\end{align}
By letting $\nu_0\downarrow0$, we find that $\phi^{I,0}(t,x)=\phi^{(0)}(t,x)$ from the property of remaining profiles along with the matching condition of initial data which means the function $\phi^{I,0}(t,x)$ is the solution to the inviscid systems \eqref{Thermo-Acoustic-System-zero}. Then, we may have $\partial_t^3\phi^{L,0}(t,z)=0$ carrying $\phi_{tt}^{L,0}(0,z)=\phi_t^{L,0}(0,z)=\phi^{L,0}(0,z)=0$. Therefore, we arrive at $\phi^{L,0}(t,z)=0$ for $t>0$.

 Next, we collect the terms with $\ml{O}(\sqrt{\nu_0})$, namely,
\begin{align*}
    &\partial_t^3\phi^{I,1}+\partial_t^3\phi^{L,1}-\gamma D_{\Th}\Delta\partial_t^2\phi^{I,1}-(1+\beta)\Delta\partial_t^2\phi^{I,0}\\
    &\qquad\qquad-\gamma c_0^2\Delta\partial_t\phi^{I,1}+(1+\beta)\gamma D_{\Th}\Delta^2\partial_t\phi^{I,0}+\gamma D_{\Th}c_0^2\Delta^2\phi^{I,1}=0.
\end{align*}
Then, by taking $\nu_0\downarrow0$, we derive the inhomogeneous inviscid equation
\begin{align}\label{WKB-1-I}
\begin{cases}
\partial_t^3\phi^{I,1}-\gamma D_{\Th}\Delta\partial_t^2\phi^{I,1}-\gamma c_0^2\Delta\partial_t\phi^{I,1}+
\gamma D_{\Th}c_0^2\Delta^2\phi^{I,1}\\ \qquad\qquad\qquad\qquad =(1+\beta)\Delta\partial_t^2\phi^{I,0}-(1+\beta)\gamma D_{\Th}\Delta^2\partial_t\phi^{I,0},&x\in\mb{R}^n,\ t>0,\\
\phi^{I,1}(0,x)=0,\ \phi_t^{I,1}(0,x)=0,\ \phi_{tt}^{I,1}(0,x)=0,&x\in\mb{R}^n,
\end{cases}
\end{align}
and 
\begin{align}\label{WKB-1-L}
\begin{cases}
\partial_t^3\phi^{L,1}=0,&z\in\mb{R}^n,\ t>0,\\
\phi^{L,1}(0,z)=0,\ \phi_t^{L,1}(0,z)=0,\ \phi_{tt}^{L,1}(0,z)=0,&z\in\mb{R}^n.
\end{cases}
\end{align}
According to the equation \eqref{WKB-1-L}, we claim that $\phi^{L,1}(t,z)=0$. Summarizing the last deduction, we conclude the next expansions.
\begin{prop}
    The solution $\phi=\phi(t,x)$ to the Cauchy problem for the thermoviscous acoustic systems \eqref{Thermo-Acoustic-System-third-equ2} with a small constant $0<\nu_0\ll 1$ formally has the following asymptotic expansions:
    \begin{align*}
        \phi(t,x)=\phi^{(0)}(t,x)+\sqrt{\nu_0}\phi^{I,1}(t,x)+\sum_{j\geqslant2}\nu_0^\frac{j}{2}\left(\phi^{I,j}(t,x)+\phi^{L,j}(t,\sqrt{\nu_0}x)\right),
    \end{align*}
    where $\phi^{(0)}=\phi^{(0)}(t,x)$ is the solution to the inviscid systems \eqref{Thermo-Acoustic-System-zero}, and $\phi^{I,1}=\phi^{I,1}(t,x)$ is the solution to the inhomogeneous Cauchy problem \eqref{WKB-1-I}. 
\end{prop}

\section*{Acknowledgments}
 The work was supported by the National Natural Science Foundation of China (grants No. 12171317, 11971497), the Natural Science Foundation of Guangdong Province (grants  No.   2023A1515012044, 2019B151502041, 2020B1515310004), and the Natural Science Foundation of the Department of Education of Guangdong Province (grants No. 2018KZDXM048, 
2019KZDXM036, 2020ZDZX3051,  2020TSZK005).

\end{document}